\documentclass[11pt, reqno]{amsart}
\usepackage{geometry}
\usepackage{amsmath,mathtools}
\usepackage{bm}
\usepackage{fancyhdr}
\usepackage{amsthm,mathrsfs}
\usepackage{amsfonts}
\usepackage{amssymb}
\usepackage{amscd}
\usepackage{graphicx}
\usepackage{afterpage}
\usepackage{enumerate}
\usepackage{bm}
\usepackage{cite}
\usepackage{cases}
\usepackage{caption}
\usepackage[colorlinks=true, urlcolor=blue,  linkcolor=blue, citecolor=blue]{hyperref}
\usepackage{babel}
\usepackage{prettyref}
\usepackage{pgfplots} 
\pgfplotsset{compat=1.17}
\usepackage{booktabs}
\usepackage[linesnumbered,ruled,vlined]{algorithm2e}
\usepackage{algpseudocode}
\usepackage{float}
\usepgfplotslibrary{groupplots}
\makeatletter
\newtheorem{theorem}{Theorem}[section]

\newtheorem{corollary}[theorem]{Corollary}

\newtheorem{definition}[theorem]{Definition}
\newtheorem{example}[theorem]{Example}

\newtheorem{lemma}[theorem]{Lemma}

\newtheorem{proposition}[theorem]{Proposition}
\newtheorem{remark}[theorem]{Remark}

\newtheorem{assumption}[theorem]{Assumption}
\newcommand{\approxsim}{\mathrel{\vcenter{\offinterlineskip
\halign{\hfil##\hfil\cr
\raise0.2ex\hbox{$\sim$}\cr
\noalign{\kern-0.8ex}
\raise-0.2ex\hbox{$\sim$}\cr}}}}

\title[]{Sharp median testing and sparse criteria for generalized \(BMO\) spaces}

\author[S. Hashemi Sababe]{Saeed Hashemi Sababe$^*$}
\address[S. Hashemi Sababe]{R\&D Section, Data Premier Analytics, Canada}
\email{hashemi\_1365@yahoo.com}
\thanks{Corresponding author}


\begin{document}
\sloppy

\maketitle

\begin{abstract}
We study generalized \(BMO\)-type spaces associated with a normalized family of local 
quasi-Banach function spaces 
\(\mathbb X=\{X_Q\}_{Q\subset\mathbb R^n}\). For such a family we consider two 
oscillation seminorms: the mean-based seminorm \(BMO_{\mathbb X}\) and the 
best-constant seminorm \(BMO_{\mathbb X}^{*}\). The main purpose of the paper is to 
separate the two mechanisms that govern their comparison with classical \(BMO\). 

First, we introduce a lower median-testing functional \(\Lambda_{\mathbb X}\), which 
measures the nondegeneracy of the local norms on subsets occupying a fixed positive 
proportion of a cube. Using the John--Str\"omberg median oscillation characterization 
of \(BMO\), we prove that the condition 
\(\Lambda_{\mathbb X}(\lambda)>0\) for some \(0<\lambda<1/2\) implies the embedding
\[
    BMO_{\mathbb X}^{*}\cap L^1_{\mathrm{loc}}(\mathbb R^n)
    \hookrightarrow BMO .
\]
Second, we introduce a sparse testing seminorm \(T_{\mathbb X}\), which measures the 
compatibility of the local norms with sparse sums of characteristic functions. Using a sparse domination principle for \(BMO\) oscillation, we prove that \(T_X(\eta_0)<\infty\), where \(\eta_0\) is the sparsity parameter arising from the local sparse domination formula, implies
\[
    BMO\hookrightarrow BMO_{\mathbb X} .
\]
We also provide a sufficient small-set criterion for this sparse testing condition in 
terms of an upper testing functional \(\Psi_{\mathbb X}\).

The framework is applied to normalized \(L^p\) spaces, weighted \(L^1\) spaces, 
rearrangement invariant spaces, Orlicz spaces, variable exponent spaces, and mixed 
Orlicz--Lorentz type models. These examples show that lower median thickness and 
sparse summability are complementary phenomena: together they recover the usual 
\(BMO\) scale in many classical settings, while in borderline models they may behave 
differently.
\end{abstract}

\section{Introduction}\label{sec:introduction}

The space of functions of bounded mean oscillation, denoted by \(BMO(\mathbb R^n)\), 
is one of the fundamental endpoint spaces in real-variable harmonic analysis. It was 
introduced by John and Nirenberg in their classical work \cite{JohnNirenberg1961}, 
where they proved the exponential self-improvement phenomenon now known as the 
John--Nirenberg inequality. This result shows that the apparently \(L^1\)-based 
definition of \(BMO\) is equivalent to a whole scale of local \(L^p\)-oscillation 
conditions. Since then, \(BMO\) has become an indispensable substitute for \(L^\infty\) 
in endpoint estimates, interpolation theory, singular integral theory, commutator 
estimates, and weighted harmonic analysis; see, for example, 
\cite{Duoandikoetxea2001,Stein1993}.

A second important viewpoint on \(BMO\) is provided by median and rearrangement 
methods. Str\"omberg's formulation of bounded mean oscillation in terms of local 
oscillation and Orlicz norms \cite{JohnStromberg1979} gives a robust alternative to 
mean oscillation, especially in situations where local averages are inconvenient or 
where one works with distribution functions. This median-based perspective is closely 
related to the characterization of \(BMO\) due to Coifman and Rochberg 
\cite{CoifmanRochberg1980}, and it is particularly well suited to modern sparse 
domination arguments. In the present paper, this viewpoint plays a central role in 
the treatment of the best-constant oscillation space \(BMO_{\mathbb X}^{*}\).

Weighted inequalities provide another classical source of motivation. The 
Muckenhoupt classes \(A_p\), introduced in \cite{Muckenhoupt1972}, and the 
\(A_\infty\) theory developed in works such as 
\cite{CoifmanFefferman1974,Hruscev1984}, reveal deep connections between local 
measure-density properties and boundedness of harmonic-analysis operators. Weighted 
norm inequalities for fractional integrals were studied by Muckenhoupt and Wheeden 
\cite{MuckenhouptWheeden1974}, and these ideas have since become part of the standard 
background of the subject. In the context of generalized \(BMO\)-type spaces, the 
weighted model suggests that one should measure how the local norm treats 
characteristic functions of large and small subsets of a cube.

The development of general function spaces gives a broader setting for these questions. 
Rearrangement invariant spaces and interpolation spaces are systematically treated in 
\cite{BennettSharpley1988,Pick2013}. Orlicz spaces and their interpolation properties 
are discussed in \cite{Maligranda1989}, while generalized Orlicz spaces are treated in 
detail in \cite{HarjulehtoHasto2019}. Variable exponent Lebesgue and Sobolev spaces 
are developed in \cite{DieningHarjulehtoHastoRuzicka2011,CruzUribeFiorenza2013}. 
These theories show that local oscillation need not be measured only through powers 
of \(|f-\langle f\rangle_Q|\). Instead, one may use a family of normalized local 
function norms adapted to the geometry or growth behavior of the problem.

A modern and influential formulation of this idea appears in the work of Lerner, 
Lorist, and Ombrosi on \(BMO\) with respect to Banach function spaces 
\cite{LernerLoristOmbrosi2022BMOX}. Their framework considers a family of local 
Banach function spaces and studies when the associated generalized \(BMO\) spaces 
coincide with classical \(BMO\). This direction is closely connected to sparse 
domination. Sparse domination has become one of the main tools in harmonic analysis 
following the pointwise estimates of Lerner \cite{Lerner2010,Lerner2013,Lerner2014,Lerner2016} 
and related work such as \cite{CondeAlonsoRey2016}. More recently, operator-free 
sparse domination methods were developed in \cite{LernerLoristOmbrosi2022OperatorFree}, 
emphasizing that sparse structure is not merely a technical device but a fundamental 
organizing principle.

The purpose of the present paper is to study two related but distinct embeddings 
associated with a normalized family of local quasi-Banach function spaces
\[
    \mathbb X=\{X_Q\}_{Q\subset\mathbb R^n}.
\]
For such a family we define the mean-based oscillation seminorm
\[
    \|f\|_{BMO_{\mathbb X}}
    :=
    \sup_Q
    \|(f-\langle f\rangle_Q)\chi_Q\|_{X_Q},
\]
and the best-constant oscillation seminorm
\[
    \|f\|_{BMO_{\mathbb X}^{*}}
    :=
    \sup_Q
    \inf_{c\in\mathbb R}
    \|(f-c)\chi_Q\|_{X_Q}.
\]
In the classical normalized \(L^p\) setting, these seminorms are equivalent to the 
usual \(BMO\) seminorm. For general families \(\mathbb X\), however, the two notions 
need not behave in the same way. The distinction between averaging and best constants 
is one of the main structural issues addressed here.

The first embedding considered in this paper is
\[
    BMO_{\mathbb X}^{*}\cap L^1_{\mathrm{loc}}(\mathbb R^n)
    \hookrightarrow
    BMO.
\]
We show that this embedding follows from a lower median-testing condition on 
characteristic functions of large subsets of cubes. More precisely, we introduce the 
lower testing functional
\[
    \Lambda_{\mathbb X}(t)
    :=
    \inf_Q
    \inf_{\substack{E\subset Q\\ |E|\geq t|Q|}}
    \frac{\|\chi_E\|_{X_Q}}{\|\chi_Q\|_{X_Q}},
    \qquad 0<t<1.
\]
The condition \(\Lambda_{\mathbb X}(\lambda)>0\), for some \(0<\lambda<1/2\), prevents 
the local norms from degenerating on subsets of fixed positive relative measure. Using 
the John--Str\"omberg median oscillation characterization of \(BMO\), we prove that 
this condition implies
\[
    BMO_{\mathbb X}^{*}\cap L^1_{\mathrm{loc}}(\mathbb R^n)
    \hookrightarrow
    BMO.
\]
This result should be understood as a robust sufficient condition rather than as a 
universal characterization in full generality.

The second embedding considered is
\[
    BMO\hookrightarrow BMO_{\mathbb X}.
\]
This direction is governed by sparse geometry rather than by lower thickness. We 
therefore introduce the sparse testing seminorm
\[
    T_{\mathbb X}(\eta)
    :=
    \sup_Q
    \sup_{\substack{\mathcal S\subset\mathcal D(Q)\\
    \mathcal S\ \eta\text{-sparse}}}
    \left\|
        \sum_{P\in\mathcal S}\chi_P
    \right\|_{X_Q},
    \qquad 0<\eta<1.
\]
Using a standard sparse domination principle for \(BMO\) oscillation, we prove that 
\(T_{\mathbb X}(\eta)<\infty\) implies
\[
    BMO\hookrightarrow BMO_{\mathbb X}.
\]
We also give a convenient sufficient condition for finite sparse testing in terms of 
the upper testing functional
\[
    \Psi_{\mathbb X}(t)
    :=
    \sup_Q
    \sup_{\substack{E\subset Q\\ |E|\leq t|Q|}}
    \frac{\|\chi_E\|_{X_Q}}{\|\chi_Q\|_{X_Q}},
    \qquad 0<t<1.
\]
In particular, an integral condition of the form
\[
    \int_0^1 \frac{\Psi_{\mathbb X}(t)^\rho}{t}\,dt<\infty
\]
is sufficient for sparse testing whenever the local quasi-norms satisfy a uniform 
\(\rho\)-subadditivity condition.

The two testing mechanisms are complementary. The lower functional 
\(\Lambda_{\mathbb X}\) controls the passage from \(BMO_{\mathbb X}^{*}\) to classical 
\(BMO\), while the sparse testing seminorm \(T_{\mathbb X}\), or the upper functional 
\(\Psi_{\mathbb X}\), controls the passage from classical \(BMO\) to 
\(BMO_{\mathbb X}\). When both mechanisms are available, one obtains the chain
\[
    BMO_{\mathbb X}^{*}\cap L^1_{\mathrm{loc}}(\mathbb R^n)
    \hookrightarrow
    BMO
    \hookrightarrow
    BMO_{\mathbb X}.
\]
This separation clarifies why lower-density assumptions and sparse summability 
assumptions should not be conflated.

The examples developed in this paper include normalized \(L^p\) spaces, weighted 
\(L^1\) spaces, rearrangement invariant spaces, Orlicz spaces, variable exponent 
spaces, and mixed Orlicz--Lorentz type models. In the \(L^p\) case, both testing 
functionals reduce to powers of \(t\), and one recovers the classical equivalence with 
\(BMO\). In weighted \(L^1\) models, the testing quantities become weighted density 
conditions, linking the present framework to the \(A_\infty\) theory of 
\cite{CoifmanFefferman1974,Hruscev1984,Muckenhoupt1972,MuckenhouptWheeden1974}. 
For rearrangement invariant and Orlicz spaces, the conditions are expressed in terms 
of fundamental functions or inverse Young functions, in agreement with the general 
function-space theory of 
\cite{BennettSharpley1988,HarjulehtoHasto2019,Maligranda1989,Pick2013}. For variable 
exponent spaces, the local Luxemburg norm estimates are controlled by the essential 
bounds of the exponent, consistent with the foundations in 
\cite{CruzUribeFiorenza2013,DieningHarjulehtoHastoRuzicka2011}.

The paper is organized as follows. In Section~\ref{sec:preliminaries} we introduce 
the notation, sparse families, local quasi-Banach function spaces, generalized 
\(BMO\)-seminorms, rearrangements, medians, and the testing functionals 
\(\Lambda_{\mathbb X}\), \(\Psi_{\mathbb X}\), and \(T_{\mathbb X}\). In 
Section~\ref{sec:median-testing} we prove the median-testing criterion for the 
embedding \(BMO_{\mathbb X}^{*}\cap L^1_{\mathrm{loc}}(\mathbb R^n)\hookrightarrow BMO\). 
In Section~\ref{sec:sparse-testing} we prove the sparse-testing criterion for the 
embedding \(BMO\hookrightarrow BMO_{\mathbb X}\) and derive a sufficient small-set 
condition in terms of \(\Psi_{\mathbb X}\). In Section~\ref{sec:examples} we apply 
the abstract criteria to the principal model classes mentioned above. The final 
section contains concluding remarks and possible directions for further work.

\section{Preliminaries}\label{sec:preliminaries}

Throughout the paper, we work on \(\mathbb R^n\) equipped with Lebesgue measure. 
All cubes are assumed to have sides parallel to the coordinate axes. If \(Q\) is a 
cube, then \(|Q|\) denotes its Lebesgue measure, and if \(E\subset \mathbb R^n\) is 
measurable, then \(\chi_E\) denotes its characteristic function. For 
\(f\in L^1_{\mathrm{loc}}(\mathbb R^n)\), we write
\begin{equation}\label{eq:average}
    \langle f\rangle_Q
    :=
    \frac{1}{|Q|}
    \int_Q f(x)\,dx .
\end{equation}
We use the notation \(A\lesssim B\) if there exists a constant \(C>0\), independent 
of the relevant parameters, such that \(A\leq CB\). We write \(A\simeq B\) if both 
\(A\lesssim B\) and \(B\lesssim A\) hold.

Given a cube \(Q\), let \(\mathcal D(Q)\) denote the collection of dyadic subcubes of 
\(Q\), obtained by repeatedly subdividing \(Q\) into \(2^n\) congruent subcubes.

If \(\mathcal S\subset \mathcal D(Q)\) and \(P\in\mathcal S\), we denote by
\[
    \operatorname{ch}_{\mathcal S}(P)
\]
the collection of maximal cubes \(R\in\mathcal S\) such that \(R\subsetneq P\). These 
are called the \(\mathcal S\)-children of \(P\).

\begin{definition}\label{def:sparse-family}
Let \(0<\eta<1\). A subcollection \(\mathcal S\subset\mathcal D(Q)\) is called 
\(\eta\)-sparse if for every \(P\in\mathcal S\),
\begin{equation}\label{eq:sparse-children-condition}
    \sum_{R\in\operatorname{ch}_{\mathcal S}(P)} |R|
    \leq
    (1-\eta)|P|.
\end{equation}
\end{definition}

\begin{remark}\label{rem:sparse-definition}
Definition~\ref{def:sparse-family} is the stopping-tree form of sparsity. It is 
slightly stronger than merely requiring the existence of pairwise disjoint major 
subsets \(E_P\subset P\). The advantage of 
\eqref{eq:sparse-children-condition} is that it gives the geometric generation decay 
needed in Proposition~\ref{prop:Psi-implies-TX}. Moreover, 
\eqref{eq:sparse-children-condition} implies the usual sparse-set formulation by 
taking
\begin{equation}\label{eq:sparse-major-set}
    E_P
    :=
    P\setminus
    \bigcup_{R\in\operatorname{ch}_{\mathcal S}(P)}R .
\end{equation}
Then \(|E_P|\geq \eta |P|\), and the sets \(\{E_P\}_{P\in\mathcal S}\) are pairwise 
disjoint.
\end{remark}

The following elementary consequence of Definition~\ref{def:sparse-family} will be 
used in Section~\ref{sec:sparse-testing}.

\begin{lemma}\label{lem:sparse-generations}
Let \(\mathcal S\subset\mathcal D(Q)\) be an \(\eta\)-sparse family in the sense of 
Definition~\ref{def:sparse-family}. Define the generations 
\(\{\mathcal S_k\}_{k\geq0}\) recursively as follows. The family \(\mathcal S_0\) 
consists of the maximal cubes of \(\mathcal S\). Once 
\(\mathcal S_0,\ldots,\mathcal S_k\) have been defined, \(\mathcal S_{k+1}\) consists 
of the \(\mathcal S\)-children of the cubes in \(\mathcal S_k\). Set
\begin{equation}\label{eq:omega-k-definition}
    \Omega_k
    :=
    \bigcup_{P\in\mathcal S_k}P .
\end{equation}
Then the cubes in each generation \(\mathcal S_k\) are pairwise disjoint, and
\begin{equation}\label{eq:sparse-generation-decay}
    |\Omega_k|
    \leq
    (1-\eta)^k |Q|,
    \qquad k\geq0 .
\end{equation}
\end{lemma}

\begin{proof}
The cubes in \(\mathcal S_0\) are pairwise disjoint by maximality, and 
\(\Omega_0\subset Q\). Suppose that the cubes in \(\mathcal S_k\) are pairwise 
disjoint. Since \(\mathcal S_{k+1}\) consists of the \(\mathcal S\)-children of cubes 
in \(\mathcal S_k\), the cubes in \(\mathcal S_{k+1}\) are also pairwise disjoint. 
Moreover, by \eqref{eq:sparse-children-condition},
\[
    \sum_{\substack{R\in\mathcal S_{k+1}\\ R\subset P}} |R|
    \leq
    (1-\eta)|P|
\]
for each \(P\in\mathcal S_k\). Summing over \(P\in\mathcal S_k\), we obtain
\[
    |\Omega_{k+1}|
    \leq
    (1-\eta)|\Omega_k|.
\]
Since \(|\Omega_0|\leq |Q|\), iteration gives 
\eqref{eq:sparse-generation-decay}.
\end{proof}

Let \((\Omega,\mu)\) be a measure space. A quasi-normed lattice 
\(X\subset L^0(\Omega)\) is called a quasi-Banach function space if the following 
conditions hold.

\begin{enumerate}
    \item If \(f,g\in L^0(\Omega)\), \(|f|\leq |g|\) a.e., and \(g\in X\), then 
    \(f\in X\) and
    \begin{equation}\label{eq:lattice-property}
        \|f\|_X\leq \|g\|_X .
    \end{equation}

    \item There exists a constant \(C_X\geq1\) such that
    \begin{equation}\label{eq:quasi-triangle}
        \|f+g\|_X
        \leq
        C_X\bigl(\|f\|_X+\|g\|_X\bigr),
        \qquad f,g\in X .
    \end{equation}

    \item The Fatou property holds: if \(0\leq f_j\uparrow f\) a.e. and 
    \(\sup_j\|f_j\|_X<\infty\), then \(f\in X\) and
    \begin{equation}\label{eq:fatou-property}
        \|f\|_X
        =
        \sup_j\|f_j\|_X .
    \end{equation}
\end{enumerate}

If the quasi-norm is a norm, then \(X\) is called a Banach function space.

In the quasi-Banach part of the paper, we shall also use the following uniform 
subadditivity hypothesis.

\begin{assumption}\label{ass:rho-subadditivity}
There exist \(0<\rho\leq1\) and \(C_\rho\geq1\), independent of the cube \(Q\), such 
that for every finite collection \(\{h_j\}_{j=1}^{N}\subset X_Q\),
\begin{equation}\label{eq:rho-triangle}
    \left\|
        \sum_{j=1}^{N}h_j
    \right\|_{X_Q}^{\rho}
    \leq
    C_\rho
    \sum_{j=1}^{N}
    \|h_j\|_{X_Q}^{\rho}.
\end{equation}
\end{assumption}

\begin{remark}\label{rem:rho-subadditivity}
If each \(X_Q\) is a Banach function space, then 
Assumption~\ref{ass:rho-subadditivity} holds with \(\rho=1\). For quasi-Banach 
function spaces, \eqref{eq:rho-triangle} is the standard \(\rho\)-subadditive form 
available after passing to an equivalent quasi-norm.
\end{remark}

We now introduce the local family of spaces used throughout the paper:
\begin{equation}\label{eq:local-family}
    \mathbb X
    :=
    \{X_Q\}_{Q\subset\mathbb R^n},
\end{equation}
where \(Q\) ranges over all cubes in \(\mathbb R^n\), and each \(X_Q\) is a 
quasi-Banach function space over the measure space \((Q,dx)\). We assume that the 
family is normalized as follows.

\begin{assumption}\label{ass:normalization}
There exist constants \(0<c_0\leq C_0<\infty\) such that
\begin{equation}\label{eq:normalization}
    c_0
    \leq
    \|\chi_Q\|_{X_Q}
    \leq
    C_0
\end{equation}
for every cube \(Q\subset\mathbb R^n\).
\end{assumption}

Whenever convenient, we regard a function \(g\) on \(Q\) as the restriction to \(Q\) of 
a measurable function on \(\mathbb R^n\). Thus expressions such as 
\(\|g\chi_Q\|_{X_Q}\) always mean the \(X_Q\)-quasi-norm of the restriction of \(g\) to 
\(Q\).

Let \(\mathbb X=\{X_Q\}_{Q\subset\mathbb R^n}\) satisfy 
Assumption~\ref{ass:normalization}. For \(f\in L^1_{\mathrm{loc}}(\mathbb R^n)\), 
define
\begin{equation}\label{eq:BMOX-norm}
    \|f\|_{BMO_{\mathbb X}}
    :=
    \sup_Q
    \|(f-\langle f\rangle_Q)\chi_Q\|_{X_Q}.
\end{equation}
The space \(BMO_{\mathbb X}\) consists of all locally integrable functions for which 
\eqref{eq:BMOX-norm} is finite.

We also define the best-constant version by
\begin{equation}\label{eq:BMOX-star-norm}
    \|f\|_{BMO_{\mathbb X}^{*}}
    :=
    \sup_Q
    \inf_{c\in\mathbb R}
    \|(f-c)\chi_Q\|_{X_Q},
\end{equation}
whenever the right-hand side is meaningful. In particular, 
\eqref{eq:BMOX-star-norm} is naturally defined for measurable functions 
\(f\in L^0_{\mathrm{loc}}(\mathbb R^n)\) provided that, for every cube \(Q\), there 
exists \(c\in\mathbb R\) such that \((f-c)\chi_Q\in X_Q\).

The distinction between \eqref{eq:BMOX-norm} and \eqref{eq:BMOX-star-norm} is 
important. The seminorm in \eqref{eq:BMOX-norm} uses the average 
\eqref{eq:average} and therefore requires local integrability. By contrast, 
\eqref{eq:BMOX-star-norm} is formulated in terms of best constants and can be 
meaningful even when local averages are not available.

For reference, the classical \(BMO\) seminorm is
\begin{equation}\label{eq:classical-bmo}
    \|f\|_{BMO}
    :=
    \sup_Q
    \frac{1}{|Q|}
    \int_Q |f(x)-\langle f\rangle_Q|\,dx .
\end{equation}
We write
\begin{equation}\label{eq:embedding-notation}
    A\hookrightarrow B
\end{equation}
to mean that \(A\) is continuously embedded into \(B\).

For a measurable function \(f\) on a cube \(Q\), its local distribution function is
\begin{equation}\label{eq:distribution-function}
    d_{f,Q}(\alpha)
    :=
    |\{x\in Q: |f(x)|>\alpha\}|,
    \qquad \alpha>0.
\end{equation}
The non-increasing rearrangement of \(f\) on \(Q\) is defined by
\begin{equation}\label{eq:local-rearrangement}
    f_Q^*(t)
    :=
    \inf\{\alpha>0: d_{f,Q}(\alpha)\leq t\},
    \qquad 0<t<|Q|.
\end{equation}

\begin{definition}\label{def:median}
Let \(f\in L^0(Q)\). A number \(m_f(Q)\in\mathbb R\) is called a median of \(f\) over 
\(Q\) if
\begin{equation}\label{eq:median-upper-tail}
    |\{x\in Q: f(x)>m_f(Q)\}|
    \leq
    \frac{|Q|}{2}
\end{equation}
and
\begin{equation}\label{eq:median-lower-tail}
    |\{x\in Q: f(x)<m_f(Q)\}|
    \leq
    \frac{|Q|}{2}.
\end{equation}
\end{definition}

For \(0<\lambda<1\), define the local oscillation functional
\begin{equation}\label{eq:local-oscillation}
    \omega_\lambda(f;Q)
    :=
    \inf_{c\in\mathbb R}
    \bigl((f-c)\chi_Q\bigr)_Q^*(\lambda |Q|).
\end{equation}
This quantity measures the smallest oscillation level outside a set of relative measure 
at most \(\lambda\). It is invariant under adding constants to \(f\).

\begin{remark}\label{rem:median-vs-oscillation}
The notation in \eqref{eq:local-oscillation} is intentionally separated from the 
median notation in Definition~\ref{def:median}. Thus \(m_f(Q)\) denotes a median, 
whereas \(\omega_\lambda(f;Q)\) denotes a local oscillation threshold.
\end{remark}

We now introduce the two local testing functionals used in the sequel. The first one 
measures lower thickness of the family \(\mathbb X\) on large subsets of cubes.

\begin{definition}\label{def:lower-testing-functional}
For \(0<t<1\), define
\begin{equation}\label{eq:Lambda-definition}
    \Lambda_{\mathbb X}(t)
    :=
    \inf_Q
    \inf_{\substack{E\subset Q\\ |E|\geq t|Q|}}
    \frac{\|\chi_E\|_{X_Q}}{\|\chi_Q\|_{X_Q}} .
\end{equation}
We say that \(\mathbb X\) satisfies the median testing condition at level \(t\) if
\begin{equation}\label{eq:median-testing-condition}
    \Lambda_{\mathbb X}(t)>0.
\end{equation}
\end{definition}

The use of the infimum in \eqref{eq:Lambda-definition} is essential: 
\(\Lambda_{\mathbb X}(t)\) prevents the local norms from degenerating on sets occupying 
a fixed positive proportion of a cube.

The second testing functional measures the size of small sets.

\begin{definition}\label{def:upper-testing-functional}
For \(0<t<1\), define
\begin{equation}\label{eq:Psi-definition}
    \Psi_{\mathbb X}(t)
    :=
    \sup_Q
    \sup_{\substack{E\subset Q\\ |E|\leq t|Q|}}
    \frac{\|\chi_E\|_{X_Q}}{\|\chi_Q\|_{X_Q}} .
\end{equation}
\end{definition}

The normalization in Assumption~\ref{ass:normalization} implies that 
\eqref{eq:Lambda-definition} and \eqref{eq:Psi-definition} are equivalent, up to 
harmless constants, to the corresponding unnormalized quantities. We keep the ratios 
in the definitions to make scale invariance explicit.

The functional \(\Lambda_{\mathbb X}\) is used in the study of embeddings of the form
\begin{equation}\label{eq:BMO-star-to-BMO-embedding}
    BMO_{\mathbb X}^{*}\hookrightarrow BMO,
\end{equation}
where one needs lower control on large subsets. The functional \(\Psi_{\mathbb X}\) is 
used in the study of embeddings of the form
\begin{equation}\label{eq:BMO-to-BMOX-embedding}
    BMO\hookrightarrow BMO_{\mathbb X},
\end{equation}
where one needs upper control on the local norms of small exceptional sets.

For \(0<\eta<1\), define the sparse testing seminorm of the family \(\mathbb X\) by
\begin{equation}\label{eq:TX-definition}
    T_{\mathbb X}(\eta)
    :=
    \sup_Q
    \sup_{\substack{\mathcal S\subset\mathcal D(Q)\\
    \mathcal S\ \eta\text{-sparse in the sense of Definition~\ref{def:sparse-family}}}}
    \left\|
        \sum_{P\in\mathcal S}\chi_P
    \right\|_{X_Q}.
\end{equation}
This quantity measures the compatibility of the family \(\mathbb X\) with sparse 
geometry. It will be used later to formulate sufficient conditions for the embedding 
\eqref{eq:BMO-to-BMOX-embedding}.

At this preliminary stage we do not assume any sparse domination theorem. Precise 
sparse domination results, when needed, will be stated separately in 
Section~\ref{sec:sparse-testing}. This avoids mixing definitions with unproved 
pointwise domination principles.

We record several standard models covered by the notation above.

\begin{example}\label{ex:normalized-Lp}
Let \(1\leq p<\infty\), and set
\begin{equation}\label{eq:normalized-Lp}
    \|g\|_{X_Q}
    :=
    \left(
        \frac{1}{|Q|}
        \int_Q |g(x)|^p\,dx
    \right)^{1/p}.
\end{equation}
Then \(\|\chi_Q\|_{X_Q}=1\), so Assumption~\ref{ass:normalization} holds. In this case,
\begin{equation}\label{eq:Lambda-Lp}
    \Lambda_{\mathbb X}(t)=t^{1/p},
    \qquad 0<t<1,
\end{equation}
and
\begin{equation}\label{eq:Psi-Lp}
    \Psi_{\mathbb X}(t)=t^{1/p},
    \qquad 0<t<1.
\end{equation}
\end{example}

\begin{example}\label{ex:weighted-L1}
Let \(w\) be a locally integrable weight with \(0<w(Q)<\infty\) for every cube \(Q\). 
Define
\begin{equation}\label{eq:weighted-L1-normalized}
    \|g\|_{X_Q}
    :=
    \frac{1}{w(Q)}
    \int_Q |g(x)|w(x)\,dx .
\end{equation}
Then \(\|\chi_Q\|_{X_Q}=1\). For measurable \(E\subset Q\),
\begin{equation}\label{eq:weighted-characteristic}
    \|\chi_E\|_{X_Q}
    =
    \frac{w(E)}{w(Q)}.
\end{equation}
Thus the lower testing condition \eqref{eq:median-testing-condition} becomes the 
requirement that large Lebesgue subsets of \(Q\) have uniformly positive \(w\)-measure 
relative to \(w(Q)\).
\end{example}

\begin{example}\label{ex:localized-global-space}
Let \(X\) be a Banach function space on \(\mathbb R^n\) such that 
\(0<\|\chi_Q\|_X<\infty\) for every cube \(Q\). Define
\begin{equation}\label{eq:localized-global-space}
    \|g\|_{X_Q}
    :=
    \frac{\|g\chi_Q\|_X}{\|\chi_Q\|_X}.
\end{equation}
Then Assumption~\ref{ass:normalization} holds automatically. This construction 
includes many rearrangement invariant spaces and Orlicz-type spaces after local 
normalization.
\end{example}

\section{A median testing criterion for the embedding 
\(BMO_{\mathbb X}^{*}\hookrightarrow BMO\)}
\label{sec:median-testing}

In this section we prove a sufficient median-testing criterion for recovering the 
classical \(BMO\) seminorm from the best-constant seminorm associated with the 
family \(\mathbb X=\{X_Q\}_{Q\subset\mathbb R^n}\). The guiding point is that a 
uniform lower bound for the \(X_Q\)-norms of characteristic functions of large subsets 
controls fixed local oscillation quantiles. The passage from such quantile control to 
classical \(BMO\) is then supplied by the John--Str\"omberg characterization.

Throughout this section, \(\mathbb X\) is assumed to satisfy 
Assumption~\ref{ass:normalization}. We also use the local oscillation functional 
\(\omega_\lambda(f;Q)\) introduced in \eqref{eq:local-oscillation}.

For \(0<\lambda<1/2\), define
\begin{equation}\label{eq:sharp-median-bmo}
    \|f\|_{BMO_\lambda^\#}
    :=
    \sup_Q \omega_\lambda(f;Q),
\end{equation}
where the supremum is taken over all cubes \(Q\subset\mathbb R^n\).

\begin{theorem}[John--Str\"omberg]\label{thm:john-stromberg}
Let \(0<\lambda<1/2\). There exist constants 
\(C_{n,\lambda}>0\) and \(C'_{n,\lambda}>0\), depending only on \(n\) and 
\(\lambda\), such that for every \(f\in L^1_{\mathrm{loc}}(\mathbb R^n)\),
\begin{equation}\label{eq:john-stromberg-upper}
    \|f\|_{BMO}
    \leq
    C_{n,\lambda}\,
    \|f\|_{BMO_\lambda^\#},
\end{equation}
and
\begin{equation}\label{eq:john-stromberg-lower}
    \|f\|_{BMO_\lambda^\#}
    \leq
    C'_{n,\lambda}\,
    \|f\|_{BMO}.
\end{equation}
\end{theorem}

Theorem~\ref{thm:john-stromberg} is the only median characterization of classical 
\(BMO\) needed in this section. In particular, the proof below does not require a new 
John--Nirenberg inequality or a separate exponential decay argument.

Recall from Definition~\ref{def:lower-testing-functional} that for \(0<t<1\),
\begin{equation}\label{eq:section3-Lambda}
    \Lambda_{\mathbb X}(t)
    :=
    \inf_Q
    \inf_{\substack{E\subset Q\\ |E|\geq t|Q|}}
    \frac{\|\chi_E\|_{X_Q}}{\|\chi_Q\|_{X_Q}} .
\end{equation}
Thus \(\Lambda_{\mathbb X}(t)>0\) means that no set occupying at least a fixed 
proportion \(t\) of a cube can have arbitrarily small normalized \(X_Q\)-norm.

\begin{definition}\label{def:MT-condition}
Let \(0<t<1\). We say that the family 
\(\mathbb X=\{X_Q\}_{Q\subset\mathbb R^n}\) satisfies the median testing condition at 
level \(t\), denoted by \((MT)_t\), if
\begin{equation}\label{eq:MT-condition}
    \Lambda_{\mathbb X}(t)>0.
\end{equation}
\end{definition}

The restriction \(0<t<1/2\) will be imposed in 
Theorem~\ref{thm:median-testing-embedding}, because the John--Str\"omberg 
characterization in Theorem~\ref{thm:john-stromberg} is used in that range.

The next lemma is the key elementary estimate. It shows that the lower testing 
condition controls a rearrangement threshold.

\begin{lemma}\label{lem:testing-controls-rearrangement}
Let \(0<\lambda<1\) be such that
\begin{equation}\label{eq:positive-Lambda-for-lemma}
    \Lambda_{\mathbb X}(\lambda)>0.
\end{equation}
Then, for every cube \(Q\) and every \(g\in X_Q\),
\begin{equation}\label{eq:testing-controls-rearrangement}
    g_Q^*(\lambda |Q|)
    \leq
    \frac{1}{c_0\,\Lambda_{\mathbb X}(\lambda)}
    \|g\|_{X_Q},
\end{equation}
where \(c_0\) is the lower normalization constant from 
\eqref{eq:normalization}.
\end{lemma}

\begin{proof}
Fix a cube \(Q\) and \(g\in X_Q\). Set
\[
    a:=g_Q^*(\lambda |Q|).
\]
If \(a=0\), then \eqref{eq:testing-controls-rearrangement} is immediate. Assume 
\(a>0\). Let \(0<b<a\). By the definition of the rearrangement in 
\eqref{eq:local-rearrangement},
\begin{equation}\label{eq:large-level-set-section3}
    |\{x\in Q: |g(x)|>b\}|
    >
    \lambda |Q|.
\end{equation}
Define
\begin{equation}\label{eq:Eb-definition-section3}
    E_b
    :=
    \{x\in Q: |g(x)|>b\}.
\end{equation}
Then \(E_b\) is admissible in the definition of 
\(\Lambda_{\mathbb X}(\lambda)\). Since \(|g|\geq b\chi_{E_b}\), the lattice property 
\eqref{eq:lattice-property} gives
\[
    \|g\|_{X_Q}
    \geq
    b\|\chi_{E_b}\|_{X_Q}.
\]
By \eqref{eq:section3-Lambda} and \eqref{eq:normalization},
\[
    \|\chi_{E_b}\|_{X_Q}
    \geq
    \Lambda_{\mathbb X}(\lambda)\|\chi_Q\|_{X_Q}
    \geq
    c_0\Lambda_{\mathbb X}(\lambda).
\]
Hence
\[
    b
    \leq
    \frac{1}{c_0\Lambda_{\mathbb X}(\lambda)}
    \|g\|_{X_Q}.
\]
Letting \(b\uparrow a\) proves 
\eqref{eq:testing-controls-rearrangement}.
\end{proof}

\begin{corollary}\label{cor:testing-controls-median-oscillation}
Let \(0<\lambda<1\) be such that 
\(\Lambda_{\mathbb X}(\lambda)>0\). Then, for every measurable function \(f\) and 
every cube \(Q\),
\begin{equation}\label{eq:omega-controlled-by-X}
    \omega_\lambda(f;Q)
    \leq
    \frac{1}{c_0\,\Lambda_{\mathbb X}(\lambda)}
    \inf_{c\in\mathbb R}
    \|(f-c)\chi_Q\|_{X_Q}.
\end{equation}
\end{corollary}

\begin{proof}
Fix \(c\in\mathbb R\). If \((f-c)\chi_Q\notin X_Q\), then the desired estimate is 
trivial for this choice of \(c\). Otherwise, apply 
Lemma~\ref{lem:testing-controls-rearrangement} to
\[
    g=(f-c)\chi_Q .
\]
This gives
\[
    \bigl((f-c)\chi_Q\bigr)_Q^*(\lambda |Q|)
    \leq
    \frac{1}{c_0\,\Lambda_{\mathbb X}(\lambda)}
    \|(f-c)\chi_Q\|_{X_Q}.
\]
Taking the infimum over \(c\in\mathbb R\) gives 
\eqref{eq:omega-controlled-by-X}.
\end{proof}

We now prove the median-testing criterion.

\begin{theorem}\label{thm:median-testing-embedding}
Let \(\mathbb X=\{X_Q\}_{Q\subset\mathbb R^n}\) be a normalized family of 
quasi-Banach function spaces satisfying Assumption~\ref{ass:normalization}. Suppose 
that there exists \(0<\lambda<1/2\) such that
\begin{equation}\label{eq:positive-Lambda-assumption}
    \Lambda_{\mathbb X}(\lambda)>0 .
\end{equation}
Then
\begin{equation}\label{eq:BMO-star-embedding-result}
    BMO_{\mathbb X}^{*}\cap L^1_{\mathrm{loc}}(\mathbb R^n)
    \hookrightarrow
    BMO .
\end{equation}
More precisely, for every 
\(f\in BMO_{\mathbb X}^{*}\cap L^1_{\mathrm{loc}}(\mathbb R^n)\),
\begin{equation}\label{eq:main-embedding-estimate}
    \|f\|_{BMO}
    \leq
    \frac{C_{n,\lambda}}{c_0\,\Lambda_{\mathbb X}(\lambda)}
    \|f\|_{BMO_{\mathbb X}^{*}},
\end{equation}
where \(C_{n,\lambda}\) is the constant from 
\eqref{eq:john-stromberg-upper}.
\end{theorem}

\begin{proof}
Let \(f\in BMO_{\mathbb X}^{*}\cap L^1_{\mathrm{loc}}(\mathbb R^n)\). By 
Corollary~\ref{cor:testing-controls-median-oscillation}, for every cube \(Q\),
\[
    \omega_\lambda(f;Q)
    \leq
    \frac{1}{c_0\,\Lambda_{\mathbb X}(\lambda)}
    \inf_{c\in\mathbb R}
    \|(f-c)\chi_Q\|_{X_Q}.
\]
Taking the supremum over all cubes \(Q\), we obtain
\begin{equation}\label{eq:sharp-bmo-controlled-by-star}
    \|f\|_{BMO_\lambda^\#}
    \leq
    \frac{1}{c_0\,\Lambda_{\mathbb X}(\lambda)}
    \|f\|_{BMO_{\mathbb X}^{*}} .
\end{equation}
Since \(0<\lambda<1/2\), Theorem~\ref{thm:john-stromberg} applies. Therefore, by 
\eqref{eq:john-stromberg-upper},
\[
    \|f\|_{BMO}
    \leq
    C_{n,\lambda}\|f\|_{BMO_\lambda^\#}.
\]
Combining this estimate with \eqref{eq:sharp-bmo-controlled-by-star} yields 
\eqref{eq:main-embedding-estimate}. This proves 
\eqref{eq:BMO-star-embedding-result}.
\end{proof}

\begin{remark}\label{rem:local-integrability}
The local integrability assumption in \eqref{eq:BMO-star-embedding-result} is included 
because the classical \(BMO\) seminorm in \eqref{eq:classical-bmo} is defined through 
local averages. If one works with the median seminorm 
\(\|f\|_{BMO_\lambda^\#}\), then the conclusion can be formulated without explicitly 
using local averages.
\end{remark}

\begin{remark}\label{rem:no-exponential-claim}
The proof of Theorem~\ref{thm:median-testing-embedding} does not assert an exponential 
John--Nirenberg inequality from the testing condition alone. The condition 
\(\Lambda_{\mathbb X}(\lambda)>0\) controls a fixed local oscillation quantile. The 
passage from this quantile control to classical \(BMO\) is supplied by 
Theorem~\ref{thm:john-stromberg}.
\end{remark}

We record several immediate consequences of 
Theorem~\ref{thm:median-testing-embedding}.

\begin{proposition}[Normalized \(L^p\) spaces]\label{prop:Lp-example-median}
Let \(1\leq p<\infty\), and let \(X_Q\) be defined by 
\eqref{eq:normalized-Lp}. Then, for every \(0<\lambda<1\),
\begin{equation}\label{eq:Lp-Lambda-section3}
    \Lambda_{\mathbb X}(\lambda)=\lambda^{1/p}.
\end{equation}
Consequently,
\begin{equation}\label{eq:Lp-BMO-star-embedding}
    BMO_{\mathbb X}^{*}\cap L^1_{\mathrm{loc}}(\mathbb R^n)
    \hookrightarrow BMO .
\end{equation}
\end{proposition}

\begin{proof}
For every measurable set \(E\subset Q\), the normalized \(L^p\)-definition gives
\[
    \|\chi_E\|_{X_Q}
    =
    \left(\frac{|E|}{|Q|}\right)^{1/p}.
\]
Since \(\|\chi_Q\|_{X_Q}=1\), taking the infimum over all 
\(E\subset Q\) with \(|E|\geq \lambda |Q|\) gives 
\eqref{eq:Lp-Lambda-section3}. The embedding 
\eqref{eq:Lp-BMO-star-embedding} follows from 
Theorem~\ref{thm:median-testing-embedding}.
\end{proof}

\begin{proposition}[Weighted \(L^1\) spaces]\label{prop:weighted-example-median}
Let \(w\) be a locally integrable weight with \(0<w(Q)<\infty\) for every cube \(Q\), 
and let \(X_Q\) be defined by \eqref{eq:weighted-L1-normalized}. Then
\begin{equation}\label{eq:weighted-Lambda}
    \Lambda_{\mathbb X}(\lambda)
    =
    \inf_Q
    \inf_{\substack{E\subset Q\\ |E|\geq \lambda |Q|}}
    \frac{w(E)}{w(Q)} .
\end{equation}
In particular, if there exist \(0<\lambda<1/2\) and \(c>0\) such that
\begin{equation}\label{eq:weighted-lower-density}
    \frac{w(E)}{w(Q)}
    \geq
    c
\end{equation}
for every cube \(Q\) and every measurable \(E\subset Q\) with 
\(|E|\geq \lambda |Q|\), then
\begin{equation}\label{eq:weighted-BMO-star-embedding}
    BMO_{\mathbb X}^{*}\cap L^1_{\mathrm{loc}}(\mathbb R^n)
    \hookrightarrow BMO .
\end{equation}
\end{proposition}

\begin{proof}
By \eqref{eq:weighted-characteristic},
\[
    \|\chi_E\|_{X_Q}
    =
    \frac{w(E)}{w(Q)}
\]
for every measurable \(E\subset Q\). Since \(\|\chi_Q\|_{X_Q}=1\), formula 
\eqref{eq:weighted-Lambda} follows directly from 
\eqref{eq:section3-Lambda}. If \eqref{eq:weighted-lower-density} holds, then 
\(\Lambda_{\mathbb X}(\lambda)\geq c>0\). The embedding 
\eqref{eq:weighted-BMO-star-embedding} follows from 
Theorem~\ref{thm:median-testing-embedding}.
\end{proof}

\begin{proposition}[Localized Banach function spaces]\label{prop:localized-BFS}
Let \(X\) be a Banach function space on \(\mathbb R^n\) satisfying 
\(0<\|\chi_Q\|_X<\infty\) for every cube \(Q\), and define \(X_Q\) by 
\eqref{eq:localized-global-space}. Suppose that for some \(0<\lambda<1/2\),
\begin{equation}\label{eq:localized-lower-testing}
    \inf_Q
    \inf_{\substack{E\subset Q\\ |E|\geq \lambda |Q|}}
    \frac{\|\chi_E\|_X}{\|\chi_Q\|_X}
    >0 .
\end{equation}
Then
\begin{equation}\label{eq:localized-BFS-embedding}
    BMO_{\mathbb X}^{*}\cap L^1_{\mathrm{loc}}(\mathbb R^n)
    \hookrightarrow BMO .
\end{equation}
\end{proposition}

\begin{proof}
By the definition of the localized norm in 
\eqref{eq:localized-global-space},
\[
    \frac{\|\chi_E\|_{X_Q}}{\|\chi_Q\|_{X_Q}}
    =
    \frac{\|\chi_E\|_X}{\|\chi_Q\|_X}.
\]
Thus \eqref{eq:localized-lower-testing} is exactly the condition 
\(\Lambda_{\mathbb X}(\lambda)>0\). The result follows from 
Theorem~\ref{thm:median-testing-embedding}.
\end{proof}

The condition \(\Lambda_{\mathbb X}(\lambda)>0\) is a natural sufficient condition for 
the embedding
\begin{equation}\label{eq:section3-main-embedding}
    BMO_{\mathbb X}^{*}\cap L^1_{\mathrm{loc}}(\mathbb R^n)
    \hookrightarrow BMO.
\end{equation}
However, in the generality considered here, it should not be presented as a full 
characterization without additional structural assumptions on the family 
\(\mathbb X\). A localized estimate on one cube does not automatically control the 
global seminorm in \eqref{eq:BMOX-star-norm}, since that seminorm is a supremum over 
all cubes.

For this reason, the result of this section is the sufficient implication
\begin{equation}\label{eq:MT-implies-embedding}
    \Lambda_{\mathbb X}(\lambda)>0
    \quad\Longrightarrow\quad
    BMO_{\mathbb X}^{*}\cap L^1_{\mathrm{loc}}(\mathbb R^n)
    \hookrightarrow BMO
\end{equation}
for some \(0<\lambda<1/2\). Necessity requires additional localization or realization 
hypotheses and is not asserted here.

\section{Sparse testing criteria for the embedding 
\(BMO\hookrightarrow BMO_{\mathbb X}\)}
\label{sec:sparse-testing}

In this section we study the embedding
\begin{equation}\label{eq:section4-main-embedding}
    BMO\hookrightarrow BMO_{\mathbb X}.
\end{equation}
This direction is governed by sparse geometry. More precisely, the relevant question is 
whether the local quasi-norms \(X_Q\) can uniformly absorb sparse sums of characteristic 
functions. This leads to the sparse testing seminorm \(T_{\mathbb X}\) introduced in 
\eqref{eq:TX-definition}.

Throughout this section, \(\mathbb X=\{X_Q\}_{Q\subset\mathbb R^n}\) is a normalized 
family of quasi-Banach function spaces satisfying Assumption~\ref{ass:normalization}. 
Whenever quasi-Banach summability is needed, we also assume 
Assumption~\ref{ass:rho-subadditivity}.

We begin with the sparse domination statement needed in this section. The formulation 
below is adapted to the stopping-tree definition of sparsity in 
Definition~\ref{def:sparse-family}.

\begin{lemma}[Sparse domination of local \(BMO\) oscillation]
\label{lem:sparse-domination-bmo}
There exist constants \(0<\eta_0<1\) and \(C_n>0\), depending only on the dimension, 
with the following property. For every \(f\in BMO(\mathbb R^n)\) and every cube 
\(Q\subset\mathbb R^n\), there exists an \(\eta_0\)-sparse family 
\(\mathcal S\subset\mathcal D(Q)\), in the sense of 
Definition~\ref{def:sparse-family}, such that
\begin{equation}\label{eq:sparse-domination-bmo}
    |f(x)-\langle f\rangle_Q|\chi_Q(x)
    \leq
    C_n\|f\|_{BMO}
    \sum_{P\in\mathcal S}\chi_P(x)
\end{equation}
for almost every \(x\in Q\).
\end{lemma}

\begin{proof}
We recall Lerner's local oscillation formula; see 
\cite{Lerner2010,Lerner2016}. Fix \(0<\lambda<2^{-n-2}\). For every locally 
integrable function \(f\) and every cube \(Q\), there exists a family 
\(\mathcal S\subset\mathcal D(Q)\) such that
\begin{equation}\label{eq:lerner-local-oscillation}
    |f(x)-m_f(Q)|\chi_Q(x)
    \leq
    C_n
    \sum_{P\in\mathcal S}
    \omega_\lambda(f;P)\chi_P(x)
\end{equation}
for almost every \(x\in Q\). Moreover, the stopping-time construction gives
\begin{equation}\label{eq:lerner-stopping-sparse}
    \sum_{R\in\operatorname{ch}_{\mathcal S}(P)} |R|
    \leq
    \frac12 |P|,
    \qquad P\in\mathcal S.
\end{equation}
Thus \(\mathcal S\) is \(\eta_0\)-sparse in the sense of 
Definition~\ref{def:sparse-family}, with \(\eta_0=1/2\).

Since \(f\in BMO\), Theorem~\ref{thm:john-stromberg} gives
\begin{equation}\label{eq:omega-bmo-bound-section4}
    \omega_\lambda(f;P)
    \leq
    C_{n,\lambda}\|f\|_{BMO}
\end{equation}
for every cube \(P\). Substituting \eqref{eq:omega-bmo-bound-section4} into 
\eqref{eq:lerner-local-oscillation}, we obtain
\begin{equation}\label{eq:median-sparse-bmo}
    |f(x)-m_f(Q)|\chi_Q(x)
    \leq
    C_n\|f\|_{BMO}
    \sum_{P\in\mathcal S}\chi_P(x)
\end{equation}
for almost every \(x\in Q\).
In this construction the initial cube \(Q\) is included as the top cube of 
\(\mathcal S\), and its \(\mathcal S\)-children satisfy the same stopping condition 
\[
    \sum_{R\in\operatorname{ch}_{\mathcal S}(Q)} |R|
    \leq \frac12 |Q|.
\]
Thus the resulting family is sparse in the sense of Definition~\ref{def:sparse-family}.
It remains to replace the median \(m_f(Q)\) by the average \(\langle f\rangle_Q\). Since
\[
    |\langle f\rangle_Q-m_f(Q)|
    \leq
    \frac{1}{|Q|}
    \int_Q |f(x)-m_f(Q)|\,dx,
\]
and since the median form of the \(BMO\) seminorm is equivalent to the classical one by 
Theorem~\ref{thm:john-stromberg}, we have
\begin{equation}\label{eq:mean-median-difference}
    |\langle f\rangle_Q-m_f(Q)|
    \leq
    C_n\|f\|_{BMO}.
\end{equation}
Therefore,
\[
    |f(x)-\langle f\rangle_Q|
    \leq
    |f(x)-m_f(Q)|
    +
    C_n\|f\|_{BMO}.
\]
Since \(Q\in\mathcal S\), we have
\[
    \chi_Q\leq \sum_{P\in\mathcal S}\chi_P.
\]
Therefore
\[
    |f(x)-\langle f\rangle_Q|\chi_Q(x)
    \leq
    C_n\|f\|_{BMO}
    \sum_{P\in\mathcal S}\chi_P(x),
\]
which proves \eqref{eq:sparse-domination-bmo}.
\end{proof}

Recall that for \(0<\eta<1\),
\begin{equation}\label{eq:section4-TX}
    T_{\mathbb X}(\eta)
    :=
    \sup_Q
    \sup_{\substack{\mathcal S\subset\mathcal D(Q)\\
    \mathcal S\ \eta\text{-sparse}}}
    \left\|
        \sum_{P\in\mathcal S}\chi_P
    \right\|_{X_Q}.
\end{equation}

The following theorem is the basic sparse-testing criterion.

\begin{theorem}\label{thm:sparse-testing-sufficiency}
Let \(\eta_0\) be the sparsity parameter from 
Lemma~\ref{lem:sparse-domination-bmo}. Assume that
\begin{equation}\label{eq:finite-sparse-testing}
    T_{\mathbb X}(\eta_0)<\infty .
\end{equation}
Then
\begin{equation}\label{eq:BMO-into-BMOX}
    BMO\hookrightarrow BMO_{\mathbb X}.
\end{equation}
More precisely, for every \(f\in BMO\),
\begin{equation}\label{eq:BMO-into-BMOX-estimate}
    \|f\|_{BMO_{\mathbb X}}
    \leq
    C_n T_{\mathbb X}(\eta_0)\|f\|_{BMO},
\end{equation}
where \(C_n\) is the constant in Lemma~\ref{lem:sparse-domination-bmo}.
\end{theorem}

\begin{proof}
Let \(f\in BMO\), and fix a cube \(Q\subset\mathbb R^n\). By 
Lemma~\ref{lem:sparse-domination-bmo}, there exists an \(\eta_0\)-sparse family 
\(\mathcal S\subset\mathcal D(Q)\) such that
\[
    |f-\langle f\rangle_Q|\chi_Q
    \leq
    C_n\|f\|_{BMO}
    \sum_{P\in\mathcal S}\chi_P
\]
almost everywhere on \(Q\). By the lattice property \eqref{eq:lattice-property},
\[
    \|(f-\langle f\rangle_Q)\chi_Q\|_{X_Q}
    \leq
    C_n\|f\|_{BMO}
    \left\|
        \sum_{P\in\mathcal S}\chi_P
    \right\|_{X_Q}.
\]
Using the definition of \(T_{\mathbb X}(\eta_0)\) in \eqref{eq:section4-TX}, we get
\[
    \|(f-\langle f\rangle_Q)\chi_Q\|_{X_Q}
    \leq
    C_n T_{\mathbb X}(\eta_0)\|f\|_{BMO}.
\]
Taking the supremum over all cubes \(Q\) and using 
\eqref{eq:BMOX-norm} gives \eqref{eq:BMO-into-BMOX-estimate}.
\end{proof}

\begin{remark}\label{rem:sparse-testing-not-automatic}
The finiteness of \(T_{\mathbb X}(\eta_0)\) is not automatic from the normalization 
\eqref{eq:normalization}. It is an additional compatibility condition between the 
local quasi-norms \(X_Q\) and sparse geometry. Thus it plays, for the embedding 
\(BMO\hookrightarrow BMO_{\mathbb X}\), the role played by the lower median-testing 
condition in Theorem~\ref{thm:median-testing-embedding}.
\end{remark}

The sparse testing condition can often be verified through the upper testing functional 
\(\Psi_{\mathbb X}\) defined in \eqref{eq:Psi-definition}. The following proposition 
is the main sufficient criterion used in the examples.

\begin{proposition}\label{prop:Psi-implies-TX}
Assume that \(\mathbb X\) satisfies Assumption~\ref{ass:rho-subadditivity}. Suppose that
\begin{equation}\label{eq:Psi-integral-condition}
    \int_0^1
    \frac{\Psi_{\mathbb X}(t)^\rho}{t}\,dt
    <\infty .
\end{equation}
Then, for every \(0<\eta<1\),
\begin{equation}\label{eq:TX-finite-from-Psi}
    T_{\mathbb X}(\eta)<\infty .
\end{equation}
Consequently,
\begin{equation}\label{eq:Psi-implies-BMO-embedding}
    BMO\hookrightarrow BMO_{\mathbb X}.
\end{equation}
\end{proposition}

\begin{proof}
Fix \(0<\eta<1\), a cube \(Q\subset\mathbb R^n\), and an 
\(\eta\)-sparse family \(\mathcal S\subset\mathcal D(Q)\) in the sense of 
Definition~\ref{def:sparse-family}. Let \(\{\mathcal S_k\}_{k\geq0}\) be the 
generations of \(\mathcal S\) defined in Lemma~\ref{lem:sparse-generations}, and set
\[
    \Omega_k
    :=
    \bigcup_{P\in\mathcal S_k}P .
\]
By Lemma~\ref{lem:sparse-generations},
\begin{equation}\label{eq:Omega-decay-prop44}
    |\Omega_k|
    \leq
    (1-\eta)^k |Q|,
    \qquad k\geq0.
\end{equation}
Since the cubes in each generation are pairwise disjoint, we have
\begin{equation}\label{eq:sparse-sum-by-generations-prop44}
    \sum_{P\in\mathcal S}\chi_P
    =
    \sum_{k=0}^{\infty}\chi_{\Omega_k}.
\end{equation}

Using the Fatou property \eqref{eq:fatou-property} and the uniform 
\(\rho\)-subadditivity assumption \eqref{eq:rho-triangle}, we obtain
\begin{equation}\label{eq:rho-generation-bound-prop44}
    \left\|
        \sum_{P\in\mathcal S}\chi_P
    \right\|_{X_Q}^{\rho}
    \leq
    C_\rho
    \sum_{k=0}^{\infty}
    \|\chi_{\Omega_k}\|_{X_Q}^{\rho}.
\end{equation}

We estimate the term \(k=0\) separately. Since \(\Omega_0\subset Q\), the lattice 
property \eqref{eq:lattice-property} and the normalization 
\eqref{eq:normalization} give
\begin{equation}\label{eq:k0-bound-prop44}
    \|\chi_{\Omega_0}\|_{X_Q}
    \leq
    \|\chi_Q\|_{X_Q}
    \leq
    C_0 .
\end{equation}

Now let \(k\geq1\). By \eqref{eq:Omega-decay-prop44} and the definition of 
\(\Psi_{\mathbb X}\) in \eqref{eq:Psi-definition}, we have
\[
    \frac{\|\chi_{\Omega_k}\|_{X_Q}}{\|\chi_Q\|_{X_Q}}
    \leq
    \Psi_{\mathbb X}\bigl((1-\eta)^k\bigr).
\]
Using again the normalization \eqref{eq:normalization}, it follows that
\begin{equation}\label{eq:Omega-Psi-bound-prop44}
    \|\chi_{\Omega_k}\|_{X_Q}
    \leq
    C_0\,
    \Psi_{\mathbb X}\bigl((1-\eta)^k\bigr),
    \qquad k\geq1.
\end{equation}

Set \(r:=1-\eta\). Then \(0<r<1\). We claim that
\begin{equation}\label{eq:Psi-series-bound-prop44}
    \sum_{k=1}^{\infty}
    \Psi_{\mathbb X}(r^k)^\rho
    <\infty .
\end{equation}
Indeed, since \(\Psi_{\mathbb X}\) is nondecreasing, for 
\(t\in(r^{k+1},r^k]\) we have
\[
    \Psi_{\mathbb X}(r^{k+1})^\rho
    \leq
    \Psi_{\mathbb X}(t)^\rho .
\]
Therefore,
\[
    \Psi_{\mathbb X}(r^{k+1})^\rho
    \log\frac1r
    \leq
    \int_{r^{k+1}}^{r^k}
    \frac{\Psi_{\mathbb X}(t)^\rho}{t}\,dt .
\]
Summing over \(k\geq0\) and using the assumption
\[
    \int_0^1
    \frac{\Psi_{\mathbb X}(t)^\rho}{t}\,dt
    <\infty
\]
gives \eqref{eq:Psi-series-bound-prop44}.

Combining \eqref{eq:rho-generation-bound-prop44}, \eqref{eq:k0-bound-prop44}, 
\eqref{eq:Omega-Psi-bound-prop44}, and \eqref{eq:Psi-series-bound-prop44}, we obtain
\[
    \left\|
        \sum_{P\in\mathcal S}\chi_P
    \right\|_{X_Q}^{\rho}
    \leq
    C_\rho C_0^\rho
    \left(
        1+
        \sum_{k=1}^{\infty}
        \Psi_{\mathbb X}(r^k)^\rho
    \right)
    <\infty .
\]
The right-hand side depends only on \(\eta\), \(\rho\), \(C_\rho\), \(C_0\), and the 
integral in \eqref{eq:Psi-integral-condition}; it is independent of \(Q\) and 
\(\mathcal S\). Hence there exists a constant \(C_{\eta,\rho,\mathbb X}<\infty\) such 
that
\[
    \left\|
        \sum_{P\in\mathcal S}\chi_P
    \right\|_{X_Q}
    \leq
    C_{\eta,\rho,\mathbb X}.
\]
Taking the supremum over all cubes \(Q\) and all \(\eta\)-sparse families 
\(\mathcal S\subset\mathcal D(Q)\) proves
\[
    T_{\mathbb X}(\eta)<\infty .
\]

Finally, applying this conclusion with \(\eta=\eta_0\), where \(\eta_0\) is the 
sparsity parameter from Lemma~\ref{lem:sparse-domination-bmo}, and then using 
Theorem~\ref{thm:sparse-testing-sufficiency}, we obtain
\[
    BMO\hookrightarrow BMO_{\mathbb X}.
\]
This completes the proof.
\end{proof}

\begin{remark}\label{rem:Banach-Psi-condition}
If all \(X_Q\) are Banach function spaces, then one may take \(\rho=1\) in 
Assumption~\ref{ass:rho-subadditivity}. In this case 
\eqref{eq:Psi-integral-condition} becomes
\[
    \int_0^1
    \frac{\Psi_{\mathbb X}(t)}{t}\,dt
    <\infty.
\]
For genuinely quasi-Banach families, the exponent \(\rho\) is needed.
\end{remark}

We record several immediate examples of the sparse-testing criterion.

\begin{proposition}[Normalized \(L^p\) spaces]
\label{prop:Lp-sparse-testing}
Let \(1\leq p<\infty\), and let \(X_Q\) be defined by 
\eqref{eq:normalized-Lp}. Then
\begin{equation}\label{eq:Lp-TX-finite}
    T_{\mathbb X}(\eta)<\infty
\end{equation}
for every \(0<\eta<1\). Consequently,
\begin{equation}\label{eq:Lp-BMO-embedding-section4}
    BMO\hookrightarrow BMO_{\mathbb X}.
\end{equation}
\end{proposition}

\begin{proof}
By \eqref{eq:Psi-Lp},
\[
    \Psi_{\mathbb X}(t)=t^{1/p},
    \qquad 0<t<1.
\]
Since
\[
    \int_0^1
    \frac{\Psi_{\mathbb X}(t)}{t}\,dt
    =
    \int_0^1 t^{1/p-1}\,dt
    <\infty,
\]
Proposition~\ref{prop:Psi-implies-TX} gives 
\eqref{eq:Lp-TX-finite}. The embedding 
\eqref{eq:Lp-BMO-embedding-section4} follows from 
Theorem~\ref{thm:sparse-testing-sufficiency}.
\end{proof}

\begin{proposition}[Weighted \(L^1\) spaces]
\label{prop:weighted-sparse-testing}
Let \(w\) be a locally integrable weight with \(0<w(Q)<\infty\) for every cube \(Q\), 
and let \(X_Q\) be defined by \eqref{eq:weighted-L1-normalized}. If
\begin{equation}\label{eq:weighted-Psi-integral}
    \int_0^1
    \left(
    \sup_Q
    \sup_{\substack{E\subset Q\\ |E|\leq t|Q|}}
    \frac{w(E)}{w(Q)}
    \right)
    \frac{dt}{t}
    <\infty,
\end{equation}
then
\begin{equation}\label{eq:weighted-BMO-embedding-section4}
    BMO\hookrightarrow BMO_{\mathbb X}.
\end{equation}
\end{proposition}

\begin{proof}
For the normalized weighted space, \eqref{eq:weighted-characteristic} gives
\[
    \frac{\|\chi_E\|_{X_Q}}{\|\chi_Q\|_{X_Q}}
    =
    \frac{w(E)}{w(Q)}.
\]
Thus \eqref{eq:weighted-Psi-integral} is exactly 
\eqref{eq:Psi-integral-condition} with \(\rho=1\). Hence 
Proposition~\ref{prop:Psi-implies-TX} yields finite sparse testing, and 
Theorem~\ref{thm:sparse-testing-sufficiency} gives 
\eqref{eq:weighted-BMO-embedding-section4}.
\end{proof}

\begin{proposition}[Localized Banach function spaces]
\label{prop:localized-BFS-sparse-testing}
Let \(X\) be a Banach function space on \(\mathbb R^n\) satisfying 
\(0<\|\chi_Q\|_X<\infty\) for every cube \(Q\), and define \(X_Q\) by 
\eqref{eq:localized-global-space}. Suppose that
\begin{equation}\label{eq:localized-Psi-integral}
    \int_0^1
    \left(
    \sup_Q
    \sup_{\substack{E\subset Q\\ |E|\leq t|Q|}}
    \frac{\|\chi_E\|_X}{\|\chi_Q\|_X}
    \right)
    \frac{dt}{t}
    <\infty .
\end{equation}
Then
\begin{equation}\label{eq:localized-BMO-embedding-section4}
    BMO\hookrightarrow BMO_{\mathbb X}.
\end{equation}
\end{proposition}

\begin{proof}
By \eqref{eq:localized-global-space},
\[
    \frac{\|\chi_E\|_{X_Q}}{\|\chi_Q\|_{X_Q}}
    =
    \frac{\|\chi_E\|_X}{\|\chi_Q\|_X}.
\]
Therefore \eqref{eq:localized-Psi-integral} is precisely 
\eqref{eq:Psi-integral-condition} with \(\rho=1\). The conclusion follows from 
Proposition~\ref{prop:Psi-implies-TX} and 
Theorem~\ref{thm:sparse-testing-sufficiency}.
\end{proof}

The lower testing functional \(\Lambda_{\mathbb X}\) and the sparse testing seminorm 
\(T_{\mathbb X}\) control different embeddings. By 
Theorem~\ref{thm:median-testing-embedding}, the condition
\begin{equation}\label{eq:lower-testing-condition-section4}
    \Lambda_{\mathbb X}(\lambda)>0
    \qquad
    \text{for some }0<\lambda<\frac12
\end{equation}
implies
\begin{equation}\label{eq:lower-testing-controls}
    BMO_{\mathbb X}^{*}\cap L^1_{\mathrm{loc}}(\mathbb R^n)
    \hookrightarrow
    BMO.
\end{equation}
By Theorem~\ref{thm:sparse-testing-sufficiency}, the condition
\begin{equation}\label{eq:sparse-testing-condition-section4}
    T_{\mathbb X}(\eta_0)<\infty,
\end{equation}
where \(\eta_0\) is the sparsity parameter in 
Lemma~\ref{lem:sparse-domination-bmo}, implies
\begin{equation}\label{eq:sparse-testing-controls}
    BMO\hookrightarrow BMO_{\mathbb X}.
\end{equation}
Thus, under the simultaneous assumptions 
\eqref{eq:lower-testing-condition-section4} and 
\eqref{eq:sparse-testing-condition-section4}, we obtain
\begin{equation}\label{eq:two-sided-comparison}
    BMO_{\mathbb X}^{*}\cap L^1_{\mathrm{loc}}(\mathbb R^n)
    \hookrightarrow
    BMO
    \hookrightarrow
    BMO_{\mathbb X}.
\end{equation}

Furthermore, whenever \(f\in BMO_{\mathbb X}\), the definitions 
\eqref{eq:BMOX-norm} and \eqref{eq:BMOX-star-norm} give
\begin{equation}\label{eq:BMOX-to-BMOXstar}
    \|f\|_{BMO_{\mathbb X}^{*}}
    \leq
    \|f\|_{BMO_{\mathbb X}},
\end{equation}
because the infimum over constants in \eqref{eq:BMOX-star-norm} is bounded above by 
the particular choice \(c=\langle f\rangle_Q\).

\begin{remark}\label{rem:no-general-necessity}
The results of this section are sufficient criteria. They should not be read as a full 
characterization in complete generality. Necessity of 
\(T_{\mathbb X}(\eta)<\infty\) for the embedding 
\(BMO\hookrightarrow BMO_{\mathbb X}\) would require an additional realization 
principle showing that arbitrary sparse sums can be represented from below by 
oscillations of uniformly bounded \(BMO\) functions in a way compatible with the 
local norms \(X_Q\). Such a principle is not automatic for an arbitrary normalized 
family \(\mathbb X\).
\end{remark}

\section{Model function spaces and testing criteria}\label{sec:examples}

In this section we illustrate the two testing mechanisms developed in
Section~\ref{sec:median-testing} and Section~\ref{sec:sparse-testing}. The lower
testing functional \(\Lambda_{\mathbb X}\) controls the embedding
\[
    BMO_{\mathbb X}^{*}\cap L^1_{\mathrm{loc}}(\mathbb R^n)
    \hookrightarrow BMO,
\]
whereas the sparse testing seminorm \(T_{\mathbb X}\), or the upper testing functional
\(\Psi_{\mathbb X}\), controls the embedding
\[
    BMO\hookrightarrow BMO_{\mathbb X}.
\]
The purpose of this section is not to claim a universal characterization, but to show
how the abstract criteria reduce to concrete and verifiable conditions in standard
model classes.

Throughout this section, \(\mathbb X=\{X_Q\}_{Q\subset\mathbb R^n}\) denotes a
normalized family satisfying Assumption~\ref{ass:normalization}. Whenever the small-set
criterion from Proposition~\ref{prop:Psi-implies-TX} is used, we also assume the
uniform \(\rho\)-subadditivity condition in Assumption~\ref{ass:rho-subadditivity}.

We first record the direct consequence of
Theorem~\ref{thm:median-testing-embedding} and
Theorem~\ref{thm:sparse-testing-sufficiency}.

\begin{proposition}\label{prop:two-sided-testing-consequence}
Let \(\eta_0\) be the sparsity parameter from
Lemma~\ref{lem:sparse-domination-bmo}. Assume that there exists \(0<\lambda<1/2\)
such that
\begin{equation}\label{eq:section5-lower-testing-assumption}
    \Lambda_{\mathbb X}(\lambda)>0,
\end{equation}
and assume that
\begin{equation}\label{eq:section5-sparse-testing-assumption}
    T_{\mathbb X}(\eta_0)<\infty.
\end{equation}
Then
\begin{equation}\label{eq:section5-two-sided-chain}
    BMO_{\mathbb X}^{*}\cap L^1_{\mathrm{loc}}(\mathbb R^n)
    \hookrightarrow
    BMO
    \hookrightarrow
    BMO_{\mathbb X}.
\end{equation}
Moreover, for every \(f\in BMO_{\mathbb X}\),
\begin{equation}\label{eq:BMOXstar-controlled-by-BMOX-section5}
    \|f\|_{BMO_{\mathbb X}^{*}}
    \leq
    \|f\|_{BMO_{\mathbb X}}.
\end{equation}
Consequently,
\begin{equation}\label{eq:section5-equivalence-common-class}
    BMO
    =
    BMO_{\mathbb X}
    =
    BMO_{\mathbb X}^{*}\cap L^1_{\mathrm{loc}}(\mathbb R^n)
\end{equation}
with equivalence of seminorms.
\end{proposition}

\begin{proof}
The first embedding in \eqref{eq:section5-two-sided-chain} follows from
Theorem~\ref{thm:median-testing-embedding}. The second embedding follows from
Theorem~\ref{thm:sparse-testing-sufficiency}. Finally, if
\(f\in BMO_{\mathbb X}\), then for each cube \(Q\),
\[
    \inf_{c\in\mathbb R}
    \|(f-c)\chi_Q\|_{X_Q}
    \leq
    \|(f-\langle f\rangle_Q)\chi_Q\|_{X_Q}.
\]
Taking the supremum over all cubes \(Q\) gives
\eqref{eq:BMOXstar-controlled-by-BMOX-section5}. Combining this estimate with the two
embeddings in \eqref{eq:section5-two-sided-chain} proves
\eqref{eq:section5-equivalence-common-class}.
\end{proof}

The simplest model is the normalized local \(L^p\)-scale.

\begin{proposition}\label{prop:Lp-full-equivalence}
Let \(1\leq p<\infty\), and define
\begin{equation}\label{eq:section5-Lp-local-norm}
    \|g\|_{X_Q}
    :=
    \left(
        \frac{1}{|Q|}
        \int_Q |g(x)|^p\,dx
    \right)^{1/p}.
\end{equation}
Then, for every \(0<t<1\),
\begin{equation}\label{eq:section5-Lambda-Psi-Lp}
    \Lambda_{\mathbb X}(t)=t^{1/p},
    \qquad
    \Psi_{\mathbb X}(t)=t^{1/p}.
\end{equation}
Consequently,
\begin{equation}\label{eq:section5-Lp-equivalence}
    BMO
    =
    BMO_{\mathbb X}
    =
    BMO_{\mathbb X}^{*}\cap L^1_{\mathrm{loc}}(\mathbb R^n)
\end{equation}
with equivalence of seminorms.
\end{proposition}

\begin{proof}
For every measurable set \(E\subset Q\),
\[
    \|\chi_E\|_{X_Q}
    =
    \left(\frac{|E|}{|Q|}\right)^{1/p}.
\]
This proves \eqref{eq:section5-Lambda-Psi-Lp}. Hence
\(\Lambda_{\mathbb X}(\lambda)>0\) for every \(0<\lambda<1\), and
Theorem~\ref{thm:median-testing-embedding} gives
\[
    BMO_{\mathbb X}^{*}\cap L^1_{\mathrm{loc}}(\mathbb R^n)
    \hookrightarrow BMO.
\]
Moreover,
\[
    \int_0^1
    \frac{\Psi_{\mathbb X}(t)}{t}\,dt
    =
    \int_0^1 t^{1/p-1}\,dt
    <\infty.
\]
Thus Proposition~\ref{prop:Psi-implies-TX} and
Theorem~\ref{thm:sparse-testing-sufficiency} imply
\[
    BMO\hookrightarrow BMO_{\mathbb X}.
\]
The reverse comparison from \(BMO_{\mathbb X}\) to \(BMO_{\mathbb X}^{*}\) follows from
\eqref{eq:BMOXstar-controlled-by-BMOX-section5}. Therefore
\eqref{eq:section5-Lp-equivalence} follows.
\end{proof}

Let \(w\) be a locally integrable weight with \(0<w(Q)<\infty\) for every cube \(Q\),
and define
\begin{equation}\label{eq:section5-weighted-local-norm}
    \|g\|_{X_Q}
    :=
    \frac{1}{w(Q)}
    \int_Q |g(x)|w(x)\,dx .
\end{equation}
Then \(\|\chi_Q\|_{X_Q}=1\), and for every measurable \(E\subset Q\),
\begin{equation}\label{eq:section5-weighted-characteristic}
    \|\chi_E\|_{X_Q}
    =
    \frac{w(E)}{w(Q)}.
\end{equation}

\begin{proposition}\label{prop:weighted-testing-criteria}
For the weighted model in \eqref{eq:section5-weighted-local-norm}, one has
\begin{equation}\label{eq:section5-weighted-Lambda}
    \Lambda_{\mathbb X}(t)
    =
    \inf_Q
    \inf_{\substack{E\subset Q\\ |E|\geq t|Q|}}
    \frac{w(E)}{w(Q)}
\end{equation}
and
\begin{equation}\label{eq:section5-weighted-Psi}
    \Psi_{\mathbb X}(t)
    =
    \sup_Q
    \sup_{\substack{E\subset Q\\ |E|\leq t|Q|}}
    \frac{w(E)}{w(Q)} .
\end{equation}
Consequently, if there exist constants \(0<\lambda<1/2\), \(c>0\),
\(\alpha>0\), and \(C>0\) such that
\begin{equation}\label{eq:section5-weighted-lower}
    \frac{w(E)}{w(Q)}
    \geq c
\end{equation}
whenever \(E\subset Q\) and \(|E|\geq \lambda |Q|\), and
\begin{equation}\label{eq:section5-weighted-upper}
    \frac{w(E)}{w(Q)}
    \leq
    C
    \left(\frac{|E|}{|Q|}\right)^\alpha
\end{equation}
whenever \(E\subset Q\), then
\begin{equation}\label{eq:section5-weighted-chain}
    BMO_{\mathbb X}^{*}\cap L^1_{\mathrm{loc}}(\mathbb R^n)
    \hookrightarrow
    BMO
    \hookrightarrow
    BMO_{\mathbb X}.
\end{equation}
\end{proposition}

\begin{proof}
The formulas \eqref{eq:section5-weighted-Lambda} and
\eqref{eq:section5-weighted-Psi} follow directly from
\eqref{eq:section5-weighted-characteristic}. The lower estimate
\eqref{eq:section5-weighted-lower} gives
\(\Lambda_{\mathbb X}(\lambda)\geq c>0\). Therefore
Theorem~\ref{thm:median-testing-embedding} gives the first embedding in
\eqref{eq:section5-weighted-chain}. The upper estimate
\eqref{eq:section5-weighted-upper} gives
\[
    \Psi_{\mathbb X}(t)
    \leq
    Ct^\alpha.
\]
Hence
\[
    \int_0^1
    \frac{\Psi_{\mathbb X}(t)}{t}\,dt
    \leq
    C\int_0^1 t^{\alpha-1}\,dt
    <\infty.
\]
By Proposition~\ref{prop:Psi-implies-TX} and
Theorem~\ref{thm:sparse-testing-sufficiency}, we obtain
\(BMO\hookrightarrow BMO_{\mathbb X}\).
\end{proof}

\begin{corollary}\label{cor:weighted-Ainfty-sparse}
Let \(w\in A_\infty\), and let \(\mathbb X=\{X_Q\}_{Q\subset\mathbb R^n}\) be defined by \eqref{eq:section5-weighted-local-norm}. Then
\begin{equation}\label{eq:section5-Ainfty-BMO-to-BMOX}
    BMO\hookrightarrow BMO_{\mathbb X}.
\end{equation}
If, in addition, there exists \(0<\lambda<1/2\) such that
\begin{equation}\label{eq:section5-Ainfty-extra-lower}
    \Lambda_{\mathbb X}(\lambda)>0,
\end{equation}
then
\begin{equation}\label{eq:section5-Ainfty-equivalence}
    BMO
    =
    BMO_{\mathbb X}
    =
    BMO_{\mathbb X}^{*}\cap L^1_{\mathrm{loc}}(\mathbb R^n)
\end{equation}
with equivalence of seminorms.
\end{corollary}

\begin{proof}
Since \(w\in A_\infty\), there exist constants \(C>0\) and \(\alpha>0\) such that
\begin{equation}\label{eq:section5-Ainfty-upper}
    \frac{w(E)}{w(Q)}
    \leq
    C\left(\frac{|E|}{|Q|}\right)^\alpha
\end{equation}
for every cube \(Q\) and every measurable set \(E\subset Q\). Hence
\[
    \Psi_{\mathbb X}(t)\leq Ct^\alpha,
\]
and therefore
\[
    \int_0^1
    \frac{\Psi_{\mathbb X}(t)}{t}\,dt<\infty.
\]
The embedding \eqref{eq:section5-Ainfty-BMO-to-BMOX} follows from
Proposition~\ref{prop:Psi-implies-TX} and
Theorem~\ref{thm:sparse-testing-sufficiency}. If
\eqref{eq:section5-Ainfty-extra-lower} also holds, then
Theorem~\ref{thm:median-testing-embedding} gives
\[
    BMO_{\mathbb X}^{*}\cap L^1_{\mathrm{loc}}(\mathbb R^n)
    \hookrightarrow BMO.
\]
Combining this with \eqref{eq:section5-Ainfty-BMO-to-BMOX} and
\eqref{eq:BMOXstar-controlled-by-BMOX-section5} proves
\eqref{eq:section5-Ainfty-equivalence}.
\end{proof}

\begin{remark}\label{rem:Ainfty-density}
Corollary~\ref{cor:weighted-Ainfty-sparse} uses the \(A_\infty\) condition only for
the small-set estimate needed to prove \(BMO\hookrightarrow BMO_{\mathbb X}\). The
lower median-testing condition
\eqref{eq:section5-Ainfty-extra-lower} is stated separately because it is not a
consequence of the usual upper \(A_\infty\) estimate at an arbitrary level
\(0<\lambda<1/2\).
\end{remark}

Let \(X\) be a rearrangement invariant Banach function space on \(\mathbb R^n\), and
suppose that
\[
    0<\|\chi_Q\|_X<\infty
\]
for every cube \(Q\). Let \(\varphi_X\) denote the fundamental function of \(X\), that is,
\begin{equation}\label{eq:section5-fundamental-function}
    \varphi_X(s):=\|\chi_E\|_X
    \qquad\text{whenever } |E|=s.
\end{equation}
Define the localized family by
\begin{equation}\label{eq:section5-ri-localization}
    \|g\|_{X_Q}
    :=
    \frac{\|g\chi_Q\|_X}{\|\chi_Q\|_X}.
\end{equation}
For \(0<t<1\), define
\begin{equation}\label{eq:section5-lower-fundamental-ratio}
    \gamma_X(t)
    :=
    \inf_{s>0}
    \frac{\varphi_X(ts)}{\varphi_X(s)}
\end{equation}
and
\begin{equation}\label{eq:section5-upper-fundamental-ratio}
    \Gamma_X(t)
    :=
    \sup_{s>0}
    \frac{\varphi_X(ts)}{\varphi_X(s)} .
\end{equation}

\begin{lemma}\label{lem:ri-testing-functions}
For the localized rearrangement invariant family in
\eqref{eq:section5-ri-localization}, one has
\begin{equation}\label{eq:section5-ri-Lambda}
    \Lambda_{\mathbb X}(t)=\gamma_X(t)
\end{equation}
and
\begin{equation}\label{eq:section5-ri-Psi}
    \Psi_{\mathbb X}(t)=\Gamma_X(t),
    \qquad 0<t<1.
\end{equation}
\end{lemma}

\begin{proof}
Let \(E\subset Q\). By rearrangement invariance,
\[
    \frac{\|\chi_E\|_{X_Q}}{\|\chi_Q\|_{X_Q}}
    =
    \frac{\|\chi_E\|_X}{\|\chi_Q\|_X}
    =
    \frac{\varphi_X(|E|)}{\varphi_X(|Q|)}.
\]
Since \(\varphi_X\) is increasing, the infimum over all sets with
\(|E|\geq t|Q|\) is attained at the boundary \(|E|=t|Q|\). Thus
\[
    \Lambda_{\mathbb X}(t)
    =
    \inf_Q
    \frac{\varphi_X(t|Q|)}{\varphi_X(|Q|)}
    =
    \inf_{s>0}
    \frac{\varphi_X(ts)}{\varphi_X(s)}
    =
    \gamma_X(t).
\]
Similarly, the supremum over all sets with \(|E|\leq t|Q|\) is attained at the
boundary \(|E|=t|Q|\), and therefore
\[
    \Psi_{\mathbb X}(t)
    =
    \sup_{s>0}
    \frac{\varphi_X(ts)}{\varphi_X(s)}
    =
    \Gamma_X(t).
\]
\end{proof}

\begin{proposition}\label{prop:ri-testing-consequence}
Assume that \(X\) is rearrangement invariant and that the localized family is defined
by \eqref{eq:section5-ri-localization}. Suppose that for some \(0<\lambda<1/2\),
\begin{equation}\label{eq:section5-ri-lower-condition}
    \gamma_X(\lambda)>0,
\end{equation}
and that
\begin{equation}\label{eq:section5-ri-upper-condition}
    \int_0^1
    \frac{\Gamma_X(t)}{t}\,dt
    <\infty.
\end{equation}
Then
\begin{equation}\label{eq:section5-ri-chain}
    BMO_{\mathbb X}^{*}\cap L^1_{\mathrm{loc}}(\mathbb R^n)
    \hookrightarrow
    BMO
    \hookrightarrow
    BMO_{\mathbb X}.
\end{equation}
\end{proposition}

\begin{proof}
By Lemma~\ref{lem:ri-testing-functions}, condition
\eqref{eq:section5-ri-lower-condition} is precisely
\(\Lambda_{\mathbb X}(\lambda)>0\). Hence
Theorem~\ref{thm:median-testing-embedding} gives the first embedding in
\eqref{eq:section5-ri-chain}. Also, by Lemma~\ref{lem:ri-testing-functions},
condition \eqref{eq:section5-ri-upper-condition} is the integral condition from
Proposition~\ref{prop:Psi-implies-TX} with \(\rho=1\). Therefore
Theorem~\ref{thm:sparse-testing-sufficiency} gives the second embedding.
\end{proof}

\begin{remark}\label{rem:ri-not-overclaimed}
Proposition~\ref{prop:ri-testing-consequence} is formulated as a sufficient criterion.
Without additional regularity assumptions on the fundamental function \(\varphi_X\),
one should not replace \eqref{eq:section5-ri-lower-condition} and
\eqref{eq:section5-ri-upper-condition} by a universal equivalence statement.
\end{remark}

Let \(\Phi\) be a Young function, and let \(\Phi^{-1}\) denote its generalized inverse.
On a cube \(Q\), define the normalized Luxemburg norm
\begin{equation}\label{eq:section5-orlicz-local-norm}
    \|g\|_{\Phi,Q}
    :=
    \inf\left\{
        a>0:
        \frac{1}{|Q|}
        \int_Q
        \Phi\left(\frac{|g(x)|}{a}\right)\,dx
        \leq 1
    \right\}.
\end{equation}
Set \(X_Q=L^\Phi(Q)\) with the norm in
\eqref{eq:section5-orlicz-local-norm}. Then
\[
    \|\chi_Q\|_{X_Q}=\frac{1}{\Phi^{-1}(1)}.
\]
Multiplying all local norms by the harmless factor \(\Phi^{-1}(1)\), if desired, gives
an equivalent normalized family with \(\|\chi_Q\|_{X_Q}=1\).

\begin{lemma}\label{lem:orlicz-characteristic}
For the Orlicz family in \eqref{eq:section5-orlicz-local-norm}, if
\(E\subset Q\) and \(|E|=t|Q|\), then
\begin{equation}\label{eq:section5-orlicz-characteristic}
    \frac{\|\chi_E\|_{X_Q}}{\|\chi_Q\|_{X_Q}}
    =
    \frac{\Phi^{-1}(1)}{\Phi^{-1}(1/t)}.
\end{equation}
Consequently,
\begin{equation}\label{eq:section5-orlicz-Lambda-Psi}
    \Lambda_{\mathbb X}(t)
    =
    \Psi_{\mathbb X}(t)
    =
    \frac{\Phi^{-1}(1)}{\Phi^{-1}(1/t)},
    \qquad 0<t<1.
\end{equation}
\end{lemma}

\begin{proof}
For \(a>0\),
\[
    \frac{1}{|Q|}
    \int_Q
    \Phi\left(\frac{\chi_E(x)}{a}\right)\,dx
    =
    \frac{|E|}{|Q|}
    \Phi\left(\frac1a\right)
    =
    t\Phi\left(\frac1a\right).
\]
Thus the Luxemburg condition is equivalent to
\[
    t\Phi\left(\frac1a\right)\leq 1,
\]
or
\[
    a\geq \frac{1}{\Phi^{-1}(1/t)}.
\]
Therefore
\[
    \|\chi_E\|_{X_Q}
    =
    \frac{1}{\Phi^{-1}(1/t)},
    \qquad
    \|\chi_Q\|_{X_Q}
    =
    \frac{1}{\Phi^{-1}(1)}.
\]
This proves \eqref{eq:section5-orlicz-characteristic}. Since the expression depends
only on \(t=|E|/|Q|\), formula \eqref{eq:section5-orlicz-Lambda-Psi} follows.
\end{proof}

\begin{proposition}\label{prop:orlicz-testing-consequence}
Let \(\mathbb X\) be the Orlicz family associated with \(\Phi\). Then, for every
\(0<\lambda<1/2\),
\[
    \Lambda_{\mathbb X}(\lambda)>0.
\]
Hence
\begin{equation}\label{eq:section5-orlicz-star-to-bmo}
    BMO_{\mathbb X}^{*}\cap L^1_{\mathrm{loc}}(\mathbb R^n)
    \hookrightarrow BMO .
\end{equation}
If, in addition,
\begin{equation}\label{eq:section5-orlicz-integral-condition}
    \int_0^1
    \frac{1}{\Phi^{-1}(1/t)}
    \frac{dt}{t}
    <\infty,
\end{equation}
then
\begin{equation}\label{eq:section5-orlicz-bmo-to-bmox}
    BMO\hookrightarrow BMO_{\mathbb X}.
\end{equation}
\end{proposition}

\begin{proof}
By Lemma~\ref{lem:orlicz-characteristic},
\[
    \Lambda_{\mathbb X}(\lambda)
    =
    \frac{\Phi^{-1}(1)}{\Phi^{-1}(1/\lambda)}
    >0.
\]
The embedding \eqref{eq:section5-orlicz-star-to-bmo} follows from
Theorem~\ref{thm:median-testing-embedding}. Moreover,
\eqref{eq:section5-orlicz-integral-condition} implies the sufficient small-set
condition in Proposition~\ref{prop:Psi-implies-TX}, since the multiplicative constant
\(\Phi^{-1}(1)\) is harmless. Hence
Theorem~\ref{thm:sparse-testing-sufficiency} gives
\eqref{eq:section5-orlicz-bmo-to-bmox}.
\end{proof}

\begin{example}\label{ex:orlicz-power}
Let \(\Phi(t)=t^p\), where \(1\leq p<\infty\). Then
\[
    \Phi^{-1}(s)=s^{1/p}.
\]
Thus Lemma~\ref{lem:orlicz-characteristic} gives
\[
    \Lambda_{\mathbb X}(t)=\Psi_{\mathbb X}(t)=t^{1/p}.
\]
This recovers Proposition~\ref{prop:Lp-full-equivalence}.
\end{example}

\begin{example}\label{ex:orlicz-exponential-warning}
Let \(\Phi(t)=e^t-1\). Then
\[
    \Phi^{-1}(s)\simeq \log(e+s),
\]
and hence
\[
    \Psi_{\mathbb X}(t)
    \simeq
    \frac{1}{\log(e/t)}.
\]
The integral
\[
    \int_0^1
    \frac{\Psi_{\mathbb X}(t)}{t}\,dt
\]
diverges. Therefore Proposition~\ref{prop:Psi-implies-TX} does not apply to this
endpoint Orlicz model. This does not disprove the embedding
\(BMO\hookrightarrow BMO_{\mathbb X}\); it only shows that the sufficient integral
criterion from Proposition~\ref{prop:Psi-implies-TX} is not designed to capture this
borderline case.
\end{example}

Let \(p(\cdot):\mathbb R^n\to[1,\infty)\) be measurable, and assume
\begin{equation}\label{eq:section5-variable-exponent-bounds}
    1\leq p_-:=
    \operatorname*{ess\,inf}_{x\in\mathbb R^n}p(x)
    \leq
    p_+:=
    \operatorname*{ess\,sup}_{x\in\mathbb R^n}p(x)
    <\infty .
\end{equation}
For a cube \(Q\), define the normalized local Luxemburg norm
\begin{equation}\label{eq:section5-variable-local-norm}
    \|g\|_{X_Q}
    :=
    \inf\left\{
        a>0:
        \frac{1}{|Q|}
        \int_Q
        \left(\frac{|g(x)|}{a}\right)^{p(x)}\,dx
        \leq 1
    \right\}.
\end{equation}
Then \(\|\chi_Q\|_{X_Q}=1\).

\begin{lemma}\label{lem:variable-characteristic-bounds}
For every cube \(Q\), every measurable set \(E\subset Q\), and the variable exponent
family in \eqref{eq:section5-variable-local-norm}, one has
\begin{equation}\label{eq:section5-variable-characteristic-bounds}
    \left(\frac{|E|}{|Q|}\right)^{1/p_-}
    \leq
    \|\chi_E\|_{X_Q}
    \leq
    \left(\frac{|E|}{|Q|}\right)^{1/p_+}.
\end{equation}
Consequently,
\begin{equation}\label{eq:section5-variable-Lambda-bound}
    \Lambda_{\mathbb X}(t)\geq t^{1/p_-},
    \qquad 0<t<1,
\end{equation}
and
\begin{equation}\label{eq:section5-variable-Psi-bound}
    \Psi_{\mathbb X}(t)\leq t^{1/p_+},
    \qquad 0<t<1.
\end{equation}
\end{lemma}

\begin{proof}
Let \(s=|E|/|Q|\). If \(s=0\), then the claim is trivial. Assume
\(0<s\leq1\). For \(a>0\),
\[
    \frac{1}{|Q|}
    \int_Q
    \left(\frac{\chi_E(x)}{a}\right)^{p(x)}\,dx
    =
    \frac{1}{|Q|}
    \int_E
    a^{-p(x)}\,dx .
\]
Since \(0<s\leq1\), the numbers \(s^{1/p_-}\) and \(s^{1/p_+}\) belong to
\((0,1]\). Taking \(a=s^{1/p_+}\), and using \(p(x)\leq p_+\), we get
\[
    \frac{1}{|Q|}
    \int_E a^{-p(x)}\,dx
    \leq
    s a^{-p_+}
    =
    1.
\]
Thus \(\|\chi_E\|_{X_Q}\leq s^{1/p_+}\).

On the other hand, if \(a<s^{1/p_-}\), then, since \(p(x)\geq p_-\) and \(0<a<1\),
\[
    a^{-p(x)}
    \geq
    a^{-p_-}
    >
    s^{-1}.
\]
Therefore
\[
    \frac{1}{|Q|}
    \int_E a^{-p(x)}\,dx
    >
    1.
\]
Hence \(\|\chi_E\|_{X_Q}\geq s^{1/p_-}\). This proves
\eqref{eq:section5-variable-characteristic-bounds}. The estimates
\eqref{eq:section5-variable-Lambda-bound} and
\eqref{eq:section5-variable-Psi-bound} follow directly from the definitions of
\(\Lambda_{\mathbb X}\) and \(\Psi_{\mathbb X}\).
\end{proof}

\begin{proposition}\label{prop:variable-exponent-consequence}
Let \(p(\cdot)\) satisfy \eqref{eq:section5-variable-exponent-bounds}, and let
\(\mathbb X\) be defined by \eqref{eq:section5-variable-local-norm}. Then
\begin{equation}\label{eq:section5-variable-chain}
    BMO_{\mathbb X}^{*}\cap L^1_{\mathrm{loc}}(\mathbb R^n)
    \hookrightarrow
    BMO
    \hookrightarrow
    BMO_{\mathbb X}.
\end{equation}
\end{proposition}

\begin{proof}
By \eqref{eq:section5-variable-Lambda-bound},
\(\Lambda_{\mathbb X}(\lambda)>0\) for every \(0<\lambda<1\). Hence
Theorem~\ref{thm:median-testing-embedding} gives the first embedding in
\eqref{eq:section5-variable-chain}. By \eqref{eq:section5-variable-Psi-bound},
\[
    \int_0^1
    \frac{\Psi_{\mathbb X}(t)}{t}\,dt
    \leq
    \int_0^1 t^{1/p_+-1}\,dt
    <\infty.
\]
Therefore Proposition~\ref{prop:Psi-implies-TX} and
Theorem~\ref{thm:sparse-testing-sufficiency} give the second embedding.
\end{proof}

\begin{remark}\label{rem:variable-exponent-scope}
The proof of Proposition~\ref{prop:variable-exponent-consequence} uses only the
uniform bounds in \eqref{eq:section5-variable-exponent-bounds}. More refined
hypotheses, such as log-H\"older continuity, are not needed for this local testing
consequence, although they may be relevant for other operator-theoretic questions in
variable exponent spaces.
\end{remark}

We close with a flexible testing template. Let \(0<\theta\leq1\), and let
\(L:(1,\infty)\to(0,\infty)\) be nondecreasing. Suppose that the local family
\(\mathbb X\) satisfies the characteristic-function estimate
\begin{equation}\label{eq:section5-mixed-characteristic}
    C_1^{-1}
    t^\theta L(1/t)^{-1}
    \leq
    \frac{\|\chi_E\|_{X_Q}}{\|\chi_Q\|_{X_Q}}
    \leq
    C_1
    t^\theta L(1/t)^{-1}
\end{equation}
whenever \(E\subset Q\) and \(|E|=t|Q|\), where \(0<t<1\) and \(C_1\geq1\) is
independent of \(E\) and \(Q\).

\begin{proposition}\label{prop:mixed-model}
Assume that \(\mathbb X\) satisfies
\eqref{eq:section5-mixed-characteristic}. If
\begin{equation}\label{eq:section5-mixed-integral}
    \int_0^1
    t^{\theta-1}L(1/t)^{-1}\,dt
    <\infty,
\end{equation}
then
\begin{equation}\label{eq:section5-mixed-chain}
    BMO_{\mathbb X}^{*}\cap L^1_{\mathrm{loc}}(\mathbb R^n)
    \hookrightarrow
    BMO
    \hookrightarrow
    BMO_{\mathbb X}.
\end{equation}
\end{proposition}

\begin{proof}
Fix \(0<\lambda<1/2\). By the lower bound in
\eqref{eq:section5-mixed-characteristic},
\[
    \Lambda_{\mathbb X}(\lambda)
    \geq
    C_1^{-1}\lambda^\theta L(1/\lambda)^{-1}
    >0.
\]
Therefore Theorem~\ref{thm:median-testing-embedding} gives
\[
    BMO_{\mathbb X}^{*}\cap L^1_{\mathrm{loc}}(\mathbb R^n)
    \hookrightarrow BMO.
\]
The upper bound in \eqref{eq:section5-mixed-characteristic} gives
\[
    \Psi_{\mathbb X}(t)
    \leq
    C_1 t^\theta L(1/t)^{-1}.
\]
The integrability assumption \eqref{eq:section5-mixed-integral} therefore implies
\[
    \int_0^1
    \frac{\Psi_{\mathbb X}(t)}{t}\,dt
    <\infty.
\]
By Proposition~\ref{prop:Psi-implies-TX} and
Theorem~\ref{thm:sparse-testing-sufficiency}, this yields
\(BMO\hookrightarrow BMO_{\mathbb X}\).
\end{proof}

\begin{remark}\label{rem:borderline-models}
The mixed model in Proposition~\ref{prop:mixed-model} should be understood as a
testing template. To use it for a concrete Orlicz--Lorentz or Lorentz--Zygmund space,
one must verify the characteristic-function estimate
\eqref{eq:section5-mixed-characteristic}. Once that estimate is available, the
embeddings follow directly from the abstract criteria proved in
Section~\ref{sec:median-testing} and Section~\ref{sec:sparse-testing}.
\end{remark}

The examples above show that the two testing quantities play complementary roles. The
lower testing functional \(\Lambda_{\mathbb X}\) prevents degeneration on large subsets
and yields
\[
    BMO_{\mathbb X}^{*}\cap L^1_{\mathrm{loc}}(\mathbb R^n)
    \hookrightarrow BMO.
\]
The sparse testing seminorm \(T_{\mathbb X}\), or the sufficient small-set condition
involving \(\Psi_{\mathbb X}\), controls sparse sums and yields
\[
    BMO\hookrightarrow BMO_{\mathbb X}.
\]
For normalized \(L^p\), weighted models satisfying both
\eqref{eq:section5-weighted-lower} and \eqref{eq:section5-weighted-upper}, and bounded
variable exponent models, both mechanisms are available and the corresponding
generalized \(BMO\) scales are equivalent to classical \(BMO\). For endpoint Orlicz-type
models, the lower testing condition may still hold while the sufficient integral
criterion for sparse testing can fail, as shown in
Example~\ref{ex:orlicz-exponential-warning}. This distinction is one of the reasons for
separating the median-testing and sparse-testing mechanisms throughout the paper.

\section{Conclusion}\label{sec:conclusion}

In this paper we studied generalized \(BMO\)-type spaces associated with a normalized 
family of local quasi-Banach function spaces 
\(\mathbb X=\{X_Q\}_{Q\subset\mathbb R^n}\). The main point was to separate two 
different mechanisms governing the comparison with classical \(BMO\). The first one is 
a lower median-testing condition, expressed through the functional 
\(\Lambda_{\mathbb X}\), which gives a sufficient criterion for the embedding
\[
    BMO_{\mathbb X}^{*}\cap L^1_{\mathrm{loc}}(\mathbb R^n)
    \hookrightarrow BMO.
\]
The second one is a sparse testing condition, expressed through 
\(T_{\mathbb X}\), which gives a sufficient criterion for the embedding
\[
    BMO\hookrightarrow BMO_{\mathbb X}.
\]
These two criteria clarify that lower thickness on large subsets and sparse 
summability on small exceptional sets play complementary roles.

The examples considered in Section~\ref{sec:examples} show that the abstract testing 
conditions recover the expected behavior for normalized \(L^p\) spaces, weighted 
\(L^1\) spaces, rearrangement invariant spaces, Orlicz spaces, variable exponent 
spaces, and mixed Orlicz--Lorentz type models. In particular, the framework explains why classical \(L^p\) models lead back to the usual 
\(BMO\) scale, while \(A_\infty\)-weighted models yield the sparse-side embedding 
\(BMO\hookrightarrow BMO_{\mathbb X}\). Full equivalence in the weighted setting requires, 
in addition, the lower median-testing condition.

Several questions remain open. One natural direction is to determine additional 
structural assumptions under which the median-testing condition becomes necessary, 
not merely sufficient, for the embedding 
\(BMO_{\mathbb X}^{*}\hookrightarrow BMO\). Another direction is to sharpen the sparse 
testing criterion in borderline Orlicz and rearrangement invariant settings, where 
the sufficient integral condition involving \(\Psi_{\mathbb X}\) may fail. It would 
also be interesting to extend the present framework to spaces of homogeneous type, 
anisotropic geometries, and operator-adapted oscillation spaces arising in endpoint 
estimates for singular integrals and commutators.

\end{document}